\documentclass[11pt,reqno]{amsart}

\usepackage{amssymb, amsmath, amsthm}
\usepackage{hyperref}
\usepackage[alphabetic,lite]{amsrefs}
\usepackage{verbatim}
\usepackage{amscd}   
\usepackage[all]{xy} 
\usepackage{youngtab} 
\usepackage{young} 
\usepackage{ytableau}
\usepackage{tikz}
\usepackage{ mathrsfs }
\usepackage{cases}
\usepackage{array}
\usepackage{tabu}
\usepackage{calligra,mathrsfs}

\newcommand{\defi}[1]{{\upshape\sffamily #1}}
\DeclareMathOperator{\ShHom}{\mathscr{H}\text{\kern -3pt {\calligra\large om}}\,}

\renewcommand{\a}{\alpha}
\renewcommand{\b}{\beta}

\newcommand{\D}{\mathcal{D}}

\newcommand{\bw}{\bigwedge}
\renewcommand{\det}{\textrm{det}}
\newcommand{\gl}{\mathfrak{gl}}

\renewcommand{\ll}{\lambda}

\newcommand{\onto}{\twoheadrightarrow}
\newcommand{\oo}{\otimes}

\newcommand{\scs}{\mathfrak{succ}}
\renewcommand{\SS}{\mathbb{S}}

\newcommand{\doub}{\widehat{AA}}

\newcommand{\x}{\underline{x}}
\newcommand{\y}{\underline{y}}
\newcommand{\z}{\underline{z}}

\newcommand{\X}{\mathcal{X}}
\newcommand{\Y}{\mathcal{Y}}

\newcommand{\Ext}{\operatorname{Ext}}
\newcommand{\Flag}{\operatorname{Flag}}
\newcommand{\GL}{\operatorname{GL}}
\newcommand{\Hom}{\operatorname{Hom}}

\newcommand{\rep}{\operatorname{rep}}
\newcommand{\rk}{\operatorname{rank}}

\newcommand{\sort}{\operatorname{sort}}

\newcommand{\Spec}{\operatorname{Spec}}
\newcommand{\Sym}{\operatorname{Sym}}
\newcommand{\Tor}{\operatorname{Tor}}

\newcommand{\add}{\operatorname{add}}

\newcommand{\coker}{\operatorname{coker}}
\renewcommand{\det}{\operatorname{det}}
\newcommand{\dom}{\operatorname{dom}}

\renewcommand{\ker}{\operatorname{ker}}
\newcommand{\opmod}{\operatorname{mod}}

\newcommand{\bb}[1]{\mathbb{#1}}

\renewcommand{\rm}[1]{\textrm{#1}}
\newcommand{\mc}[1]{\mathcal{#1}}
\newcommand{\mf}[1]{\mathfrak{#1}}
\newcommand{\ol}[1]{\overline{#1}}
\newcommand{\Qp}[2]{\mathcal{Q}_{#1}(#2)}
\newcommand{\tl}[1]{\tilde{#1}}
\newcommand{\ul}[1]{\underline{#1}}
\newcommand{\scpr}[2]{\left\langle #1,#2 \right\rangle}

\def\lra{\longrightarrow}

\newtheorem{theorem}{Theorem}[section]
\newtheorem*{theorem*}{Theorem}
\newtheorem*{problem*}{Problem}
\newtheorem{lemma}[theorem]{Lemma}

\newtheorem{proposition}[theorem]{Proposition}
\newtheorem{corollary}[theorem]{Corollary}
\newtheorem*{corollary*}{Corollary}

\newtheorem*{main-thm*}{Main Theorem}
\newtheorem*{linear-resolutions*}{Theorem on Linear Resolutions}
\newtheorem*{regularity-powers*}{Theorem on Regularity}
\newtheorem*{injectivity-Ext*}{Theorem on Injectivity of Maps of Ext Modules}
\newtheorem*{Kodaira*}{Kodaira Vanishing for Determinantal Thickenings}

\theoremstyle{definition}

\newtheorem*{definition*}{Definition}
\newtheorem{example}[theorem]{Example}

\theoremstyle{remark}
\newtheorem{remark}[theorem]{Remark}
\newtheorem*{remark*}{Remark}

\numberwithin{equation}{section}

\begin{document}

\title{Iterated local cohomology groups and Lyubeznik numbers for determinantal rings}

\author{Andr\'as C. L\H{o}rincz}
\address{Department of Mathematics, Purdue University, West Lafayette, IN 47907}
\email{alorincz@purdue.edu}

\author{Claudiu Raicu}
\address{Department of Mathematics, University of Notre Dame, 255 Hurley, Notre Dame, IN 46556\newline
\indent Institute of Mathematics ``Simion Stoilow'' of the Romanian Academy}
\email{craicu@nd.edu}


\date{\today}


\begin{abstract} 
 We give an explicit recipe for determining iterated local cohomology groups with support in ideals of minors of a generic matrix in characteristic zero, expressing them as direct sums of indecomposable $\D$-modules. For non-square matrices these indecomposables are simple, but this is no longer true for square matrices where the relevant indecomposables arise from the pole order filtration associated with the determinant hypersurface. Specializing our results to a single iteration, we determine the Lyubeznik numbers for all generic determinantal rings, thus answering a question of Hochster. 
\end{abstract}

\maketitle

\section{Introduction}\label{sec:intro}

We consider positive integers $m\geq n\geq 1$ and let $X = \bb{C}^{m\times n}$ denote the affine space of $m\times n$ complex matrices, equipped with the natural action of the group $\GL = \GL_m(\bb{C}) \times \GL_n(\bb{C})$. We denote the orbits of the $\GL$-action by $O_p$, $0\leq p\leq n$, where $O_p$ consists of matrices of rank $p$, and write $H^{\bullet}_{\ol{O}_p}(-)$ for the functors of \defi{local cohomology with support in the orbit closures}. If we let $S = \bb{C}[x_{ij}]$ denote the coordinate ring of $X$, and let $I_{p+1}$ be the ideal of $(p+1)\times (p+1)$ minors of the matrix of indeterminates $(x_{ij})$, then $I_{p+1}$ is the ideal of functions vanishing on the variety $\ol{O}_p$, and the functors $H^{\bullet}_{\ol{O}_p}(-)$ are often denoted by $H^{\bullet}_{I_{p+1}}(-)$, and referred to as the functors of \defi{local cohomology with support in the ideal $I_{p+1}$}. The goal of this work is to give an explicit recipe for computing all the \defi{iterated local cohomology groups}
\begin{equation}\label{eq:iterated-loccoh}
 H_{\ol{O}_{i_1}}^{\bullet} ( H_{\ol{O}_{i_2}}^{\bullet}(\cdots H_{\ol{O}_{i_r}}^{\bullet}(S)\cdots)).
\end{equation}
Specializing our results to the case $H_{\ol{O}_0}^{\bullet}(H_{\ol{O}_p}^{\bullet}(S))$ we determine the \defi{Lyubeznik numbers} of the coordinate ring of each $\ol{O}_p$, and observe a dichotomy between the case of square and non-square matrices. This is explained geometrically by the way the conormal varieties to the orbits intersect in the two cases, and algebraically by the fact that an appropriate category of modules is semi-simple for non-square matrices, and quite interesting for square matrices.

The groups (\ref{eq:iterated-loccoh}) are finitely generated modules over the \defi{Weyl algebra $\D_X$} of differential operators on $X$, which in addition are \defi{equivariant} for the action of the group $\GL$. We will therefore work in the category $\opmod_{\GL}(\D_X)$ of $\GL$-equivariant $\D_X$-modules, which is known by a result of Vilonen \cite[Theorem~4.3]{vilonen} to be equivalent to the category of finitely generated modules over a finite dimensional algebra, or alternatively, to the category of finite dimensional representations of a quiver with relations. The explicit description of the relevant quiver has been obtained in \cite[Theorem~4.4]{lor-wal}, and it is closely related to that of the quiver attached to a slightly larger category considered in \cite[Section~4.1]{braden-grinberg}. We identify a suitable finite set of indecomposable objects in $\opmod_{\GL}(\D_X)$ and express each of the local cohomology groups in (\ref{eq:iterated-loccoh}) as a direct sum of these indecomposables. The multiplicities of indecomposables are encoded in terms of Gaussian binomial coefficients (reviewed in Section~\ref{subsec:binomials}). Our proofs employ the symmetries coming from the $\GL$-action, the inductive structure of determinantal varieties, and the quiver description of $\opmod_{\GL}(\D_X)$, as well as a number of vanishing results for local cohomology that we prove by working on appropriate desingularizations of determinantal varieties, and using Grothendieck duality and the Borel--Weil--Bott theorem.

For non-square matrices ($m>n$) the category $\opmod_{\GL}(\D_X)$ is semi-simple by \cite[Theorem~6.7]{macpherson-vilonen}, since the conormal varieties to the orbits (described in \cite{strickland}) intersect in codimension $\geq 2$. This has two important implications:
\begin{itemize}
 \item The indecomposable modules in $\opmod_{\GL}(\D_X)$ are simple.
 \item The module structure of $M\in\opmod_{\GL}(\D_X)$ is determined up to isomorphism by its class $[M]_{\D}$ in the \defi{Grothendieck group $\Gamma_{\D}$ of $\opmod_{\GL}(\D_X)$} (see Section~\ref{subsec:mod-GL-DX}).
\end{itemize}
For this reason we begin by considering the simpler problem of determining the class in $\Gamma_{\D}$ of a local cohomology group. We return to the general case $m\geq n$ and let
\[D_0,D_1,\cdots,D_n\]
denote the simple objects in $\opmod_{\GL}(\D_X)$, where $D_p$ has support equal to $\ol{O}_p$, and is often referred to as the \defi{intersection homology $\D_X$-module} corresponding to the orbit $O_p$. When $p=n$, we have that $\ol{O}_n=X$ and $D_n=S$ is the coordinate ring of $X$. Our first theorem determines the class in $\Gamma_{\D}$ of the local cohomology groups of each $D_p$, thus generalizing the main result of \cite{raicu-weyman} which addresses the case $p=n$. 

\begin{theorem}\label{thm:Grothendieck-group}
For every $0\leq t<p\leq n\leq m$ we have the following equality in $\Gamma_{\D}[q]$:
\begin{equation}\label{eq:thm1}
\sum_{j\geq 0} [H^j_{\ol{O}_t}(D_p)]_{\D} \cdot q^j = \sum_{s=0}^t [D_s]_{\D} \cdot q^{(p-t)^2+(p-s)\cdot(m-n)} \cdot {n-s \choose p-s}_{q^2} \cdot {p-1-s \choose t-s}_{q^2}.
\end{equation}
\end{theorem}

The restriction to the case $t<p$ is done in order to avoid trivialities. If $M$ is a module whose support is contained in $\ol{O}_t$ (such as $M=D_p$ or $M=H^j_{\ol{O}_p}(N)$ for $p\leq t$, $j\geq 0$, and any module $N$) then
\begin{equation}\label{eq:suppM-in-Ot}
H^0_{\ol{O}_t}(M) = M\mbox{ and }H^i_{\ol{O}_t}(M)=0\mbox{ for }i>0.
\end{equation}
For this reason, there is no harm in assuming for instance that $i_1<i_2<\cdots<i_r$ in (\ref{eq:iterated-loccoh}).

\begin{example}\label{ex:GammaD-3x2}
 Consider the case when $m=3$ and $n=2$. For $p=2$ and $t=1$ we have $D_2=S$ and
 \[\sum_{j\geq 0} [H^j_{\ol{O}_1}(S)]_{\D} \cdot q^j \overset{(\ref{eq:thm1})}{=} \sum_{s=0}^1 [D_s]_{\D} \cdot q^{3-s}=[D_1]_{\D}\cdot q^2 + [D_0]_{\D}\cdot q^3,\]
 which yields the perhaps familiar statement that the only non-zero local cohomology groups are in this case
 \[H^2_{\ol{O}_1}(S) = D_1\mbox{ and }H^3_{\ol{O}_1}(S) = D_0.\]
 For $p=1$ and $t=0$ we obtain
 \[\sum_{j\geq 0} [H^j_{\ol{O}_0}(D_1)]_{\D} \cdot q^j \overset{(\ref{eq:thm1})}{=} [D_0]_{\D} \cdot q^2\cdot{2\choose 1}_{q^2}=[D_0]_{\D}\cdot q^2 + [D_0]_{\D}\cdot q^4.\]
Combining this with the observation (\ref{eq:suppM-in-Ot}) it follows that the only non-zero groups $H^\bullet_{\ol{O}_0}(H^\bullet_{\ol{O}_1}(S))$ are
\begin{equation}\label{eq:iter-loccoh-3x2}
H^2_{\ol{O}_0}(H^2_{\ol{O}_1}(S)) = H^4_{\ol{O}_0}(H^2_{\ol{O}_1}(S)) = H^0_{\ol{O}_0}(H^3_{\ol{O}_1}(S)) = D_0.
\end{equation}
\end{example}

Iterated local cohomology groups have been studied in the seminal work of Lyubeznik \cite{lyubeznik}, where he introduced a new set of numerical invariants attached to any local ring which is a quotient of a regular local ring containing a field \cite[Theorem-Definition~4.1]{lyubeznik}. These invariants are known today under the name of \defi{Lyubeznik numbers}, and have been the subject of extensive investigation (see \cite{lyub-survey} and the references therein). For determinantal rings, the question of describing the Lyubeznik numbers was posed by Mel Hochster as part of his list of ``Thirteen Open Questions about Local Cohomology". Part of our work here is dedicated to answering this question. For $p<n$ we let $S/I_{p+1}$ denote the coordinate ring of $\ol{O}_p$ and let $R^{(p)}=(S/I_{p+1})_{\mf{m}}$ denote its localization at the maximal homogeneous ideal. The \defi{Lyubeznik numbers $\ll_{i,j}(R^{(p)})$} are characterized by the equalities
\begin{equation}\label{eq:def-lyub-nos}
 H^i_{\ol{O}_0}(H^{m\cdot n-j}_{\ol{O}_p}(S)) = D_0^{\oplus\ll_{i,j}(R^{(p)})}.
\end{equation}
We encode the Lyubeznik numbers of determinantal rings by a bivariate generating function $L_p(q,w)\in\bb{Z}[q,w]$,
\begin{equation}
L_p(q,w) = \sum_{i,j\geq 0} \ll_{i,j}(R^{(p)}) \cdot q^i \cdot w^j.
\end{equation}
We prefer this encoding since it is more compact than the one given by the \defi{Lyubeznik tables} 
\[\Lambda(R^{(p)})=\left(\ll_{i,j}(R^{(p)})\right)_{0\leq i,j\leq\dim(R^{(p)})}\]
as defined in \cite[Definition~3.3]{lyub-survey}. We have for instance from~(\ref{eq:iter-loccoh-3x2}) that when $m=3$ and $n=2$
\[
L_1(q,w) = w^3 + q^2\cdot w^4 + q^4\cdot w^4,\quad\mbox{ or equivalently}\quad\Lambda(R^{(1)}) = \begin{pmatrix} 
0 & 0 & 0 & 1 & 0 \\
0 & 0 & 0 & 0 & 0 \\
0 & 0 & 0 & 0 & 1 \\
0 & 0 & 0 & 0 & 0 \\
0 & 0 & 0 & 0 & 1
\end{pmatrix}
\]
In this example, $R^{(1)}$ is the local ring at the vertex of the affine cone of the Segre embedding $\bb{P}^1\times\bb{P}^2\lra\bb{P}^5$. Since $\bb{P}^1\times\bb{P}^2$ is smooth, it is known that the Lyubeznik numbers have a topological interpretation, being determined by the Betti numbers of $\bb{P}^1\times\bb{P}^2$ \cites{garcia-lopez-sabbah,switala}. By contrast, there are singular (reducible) examples where the Lyubeznik numbers at the cone point depend on the projective embedding \cite{reichelt-saito-walther}, so the topology does not control on its own the Lyubeznik numbers. For non-square matrices our Theorem~\ref{thm:Grothendieck-group}, together with the fact that $\opmod_{\GL}(\D_X)$ is semi-simple, gives the following description of Lyubeznik numbers.

\begin{theorem}\label{thm:lyub-non-square}
If $m>n>p$ then the Lyubeznik numbers for $R^{(p)}$ are computed by
\begin{equation}\label{eq:lyub-non-square}
 L_p(q,w) = \sum_{s=0}^p q^{s^2 + s\cdot(m-n)} \cdot {n\choose s}_{q^2} \cdot w^{p^2 + 2 p + s\cdot(m+n-2 p-2)} \cdot {n-1-s\choose p-s}_{w^2}.
\end{equation}
\end{theorem}

In fact, using Theorem~\ref{thm:Grothendieck-group} and the semi-simplicity of $\opmod_{\GL}(\D_X)$ we can determine (\ref{eq:iterated-loccoh}), and in particular describe the \defi{generalized Lyubeznik numbers} as defined in \cite[Section~7]{lyub-survey}. More generally,
\[ H_{\ol{O}_{i_1}}^{\bullet} ( H_{\ol{O}_{i_2}}^{\bullet}(\cdots H_{\ol{O}_{i_r}}^{\bullet}(D_p)\cdots))\]
can be computed for any $D_p$. We leave the determination of the precise formulas to the interested reader. 

When $m=n$ the situation is more subtle, as can be seen already in the following simple example.
\begin{example}\label{ex:2x2}
 Suppose that $m=n=2$ and let $p=1$. Applying (\ref{eq:thm1}) we get 
 \[[H^1_{\ol{O}_1}(S)]_{\D} = [D_0]_{\D} + [D_1]_{\D}\]
 but $H^1_{\ol{O}_1}(S)$ is not the direct sum of $D_0$ and $D_1$! If we write $\det$ for the $2\times 2$ determinant, then $H^1_{\ol{O}_1}(S)=S_{\det}/S$ contains no non-zero elements annihilated by the maximal homogeneous ideal, so it can't contain $D_0$ (which is supported at $0$) as a submodule. This observation is also reflected in the calculation of Lyubeznik numbers, as follows. Since $\ol{O}_1$ is a hypersurface of (affine) dimension $3$ (the cone over a quadric in $\bb{P}^3$), the only non-zero Lyubeznik number is $\ll_{3,3}(R^{(1)})=1$, that is the only non-zero group $H^\bullet_{\ol{O}_0}(H^\bullet_{\ol{O}_1}(S))$ is
 \[H^3_{\ol{O}_0}(H^1_{\ol{O}_1}(S)) = D_0.\]
The non-zero local cohomology groups $H^\bullet_{\ol{O}_0}(D_0)$ and $H^\bullet_{\ol{O}_0}(D_1)$ are by (\ref{eq:thm1}) and (\ref{eq:suppM-in-Ot})
\[H^0_{\ol{O}_0}(D_0) = H^1_{\ol{O}_0}(D_1) = H^3_{\ol{O}_0}(D_1) = D_0,\]
so the local cohomology groups of $H^1_{\ol{O}_1}(S)$ are not the direct sums of those of $D_0$ and $D_1$. In particular, specializing (\ref{eq:lyub-non-square}) to the case when $m=n$ would give the wrong answer! Instead, we have the following.
\end{example}

\begin{theorem}\label{thm:lyub-square}
If $m=n$ then $L_{n-1}(q,w)=(q\cdot w)^{n^2-1}$ and for $0\leq p\leq n-2$ we have
\begin{equation}\label{eq:lyub-square}
 L_p(q,w) = \sum_{s=0}^p q^{s^2 + 2 s} \cdot {n-1\choose s}_{q^2} \cdot w^{p^2+2 p + s\cdot(2 n-2 p-2)} \cdot {n-2-s\choose p-s}_{w^2}.
\end{equation}
\end{theorem}
For instance, in the case of $4\times 4$ matrices of rank at most $2$ ($m=n=4$ and $p=2$) we obtain
\begin{equation}\label{eq:L2-4x4}
L_2(q,w) = w^8 + (q^3+q^5+q^7)\cdot w^{10} + (q^8 + q^{10} + q^{12})\cdot w^{12}.
\end{equation}

As we saw in Example~\ref{ex:2x2}, for square matrices the (iterated) local cohomology groups of $S$ are no longer expressible as direct sums of the simple modules $D_p$. We proceed instead to construct a different set of indecomposables that play the role of the simples. We let $\det=\det(x_{ij})$ denote the determinant of the generic $n\times n$ matrix, and let $\langle \det^{-p}\rangle_{\D}$ denote the $\D_X$-submodule of $S_{\det}$ generated by $\det^{-p}$. It is shown in \cite[Theorem~1.1]{raicu-dmods} that
 \begin{equation}\label{eq:filtration-Sdet}
 0\subsetneq S\subsetneq\langle\det^{-1}\rangle_{\D}\subsetneq\cdots\subsetneq\langle\det^{-n}\rangle_{\D}=S_{\det}
 \end{equation}
 is a $\D_X$-module composition series with composition factors $S\simeq D_n$ and 
 \begin{equation}\label{eq:quots-filt-Sdet}
 \frac{\langle \det^{-p}\rangle_{\D}}{\langle \det^{-p+1}\rangle_{\D}} \simeq D_{n-p}\mbox{ for }p=1,\cdots,n.
 \end{equation}
We define $Q_n=S_{\det}$ and for $p=0,\cdots,n-1$, we let 
\begin{equation}\label{eq:def-Qp}
Q_p = \frac{S_{\det}}{\langle\det^{p-n+1}\rangle_{\D}}.
\end{equation}
It follows from (\ref{eq:filtration-Sdet}) and (\ref{eq:quots-filt-Sdet}) that $Q_p$ has composition factors $D_0,\cdots,D_p$, hence
\begin{equation}\label{eq:Qp=sum-Ds}
 [Q_p]_{\D} = \sum_{s=0}^p [D_s]_{\D}
\end{equation}
and the support of $Q_p$ is $\ol{O}_p$. We denote by $\add(Q)$ the additive subcategory of $\opmod_{\GL}(\D_X)$ consisting of modules that are isomorphic to a direct sum of copies of $Q_0,Q_1,\dots, Q_n$. It follows from (\ref{eq:Qp=sum-Ds}) that $[Q_0]_{\D},\cdots,[Q_n]_{\D}$ form a basis of the Grothendieck group $\Gamma_{\D}$, so a module $M\in\add(Q)$ is determined up to isomorphism by $[M]_{\D}$. The following result (when combined with (\ref{eq:thm1}), (\ref{eq:suppM-in-Ot}), and (\ref{eq:Qp=sum-Ds})) allows one to determine (\ref{eq:iterated-loccoh}) when $m=n$, or more generally to describe arbitrary iterations
\[ H_{\ol{O}_{i_1}}^{\bullet} ( H_{\ol{O}_{i_2}}^{\bullet}(\cdots H_{\ol{O}_{i_r}}^{\bullet}(M)\cdots))\mbox{ where }M=D_p\mbox{ or }M=Q_p,\ p=0,\cdots,n.\]

\begin{theorem}\label{thm:2}
For every $0\leq t<p\leq n=m$ and $j\geq 0$ we have that 
\[H^j_{\ol{O}_t}(D_p)\in\add(Q)\mbox{ and }H^j_{\ol{O}_t}(Q_p)\in\add(Q).\]
Moreover,
\begin{equation}\label{eq:thm2}
\sum_{j\geq 0} [H^j_{\ol{O}_t}(Q_p)]_{\D} \cdot q^j = \sum_{s=0}^t [Q_s]_{\D} \cdot q^{(p-t)^2 + 2(p-s)} \cdot {n-s-1\choose p-s}_{q^2} \cdot {p-s-1\choose p-t-1}_{q^2}.
\end{equation}
\end{theorem}

This theorem is explained in Section~\ref{sec:module-structure}. A formula analogous to (\ref{eq:thm2}) holds for the groups $H^j_{\ol{O}_t}(D_p)$, and can be obtained based on (\ref{eq:thm1}) from the fact that $H^j_{\ol{O}_t}(D_p)\in\add(Q)$ (see Theorem~\ref{thm:square-loccoh-Dp}). To see how Theorem~\ref{thm:2} allows for the calculation of Lyubeznik numbers, or more general iterated local cohomology groups, we explain next how to derive (\ref{eq:L2-4x4}).

\begin{example}\label{ex:L2-4x4}
 If $m=n=4$ and $p=2$ then we have
 \[
 \begin{aligned}
 \sum_{j\geq 0} [H^j_{\ol{O}_2}(S)]_{\D} \cdot q^j &\overset{(\ref{eq:thm1})}{=} [D_0]_{\D}\cdot(q^4+q^6+q^8)+[D_1]_{\D}\cdot(q^4+q^6)+[D_2]_{\D}\cdot q^4 \\
 &\overset{(\ref{eq:Qp=sum-Ds})}{=} [Q_2]_{\D}\cdot q^4 + [Q_1]_{\D}\cdot q^6 + [Q_0]_{\D}\cdot q^8. 
 \end{aligned}
 \]
 By Theorem~\ref{thm:2} we have that $H^j_{\ol{O}_2}(S)\in\add(Q)$ for all $j$, hence 
 \[H^4_{\ol{O}_2}(S) = Q_2,\ H^6_{\ol{O}_2}(S) = Q_1,\mbox{ and }H^8_{\ol{O}_2}(S) = Q_0 = D_0.\] 
 Using (\ref{eq:suppM-in-Ot}) we get $H^0_{\ol{O}_0}(H^8_{\ol{O}_2}(S)) = D_0$ and therefore $\ll_{0,8}(R^{(2)})=1$. Using (\ref{eq:thm2}) we have
 \[\sum_{j\geq 0} [H^j_{\ol{O}_0}(Q_1)]_{\D} \cdot q^j = [Q_0]_{\D} \cdot q^3\cdot{3\choose 1}_{q^2} = [D_0]_{\D} \cdot (q^3+q^5+q^7)\]
 and therefore $\ll_{3,10}(R^{(2)}) = \ll_{5,10}(R^{(2)}) = \ll_{7,10}(R^{(2)}) = 1$. Using (\ref{eq:thm2}) again we have
 \[\sum_{j\geq 0} [H^j_{\ol{O}_0}(Q_2)]_{\D} \cdot q^j = [Q_0]_{\D} \cdot q^8\cdot{3\choose 2}_{q^2} = [D_0]_{\D} \cdot (q^8+q^{10}+q^{12})\]
 and therefore $\ll_{8,12}(R^{(2)}) = \ll_{10,12}(R^{(2)}) = \ll_{12,12}(R^{(2)}) = 1$. All the remaining Lyubeznik numbers vanish, proving (\ref{eq:L2-4x4}).
\end{example}

The paper is organized as follows. In Section~\ref{sec:prelim} we recall some basic notions regarding weights and Schur functors, $q$-binomial coefficients, categories of admissible representations and equivariant $\D$-modules, and Bott's theorem for Grassmannians and flag varieties. We also discuss briefly families of determinantal rings over a general base, and the inductive structure of determinantal rings. In Section~\ref{sec:loccoh-in-Gamma} we prove Theorems~\ref{thm:Grothendieck-group} and~\ref{thm:lyub-non-square}. Sections~\ref{sec:vanishing-loccoh-Jzl} and \ref{sec:H1-vanishing} are concerned with a number of technical results proving the vanishing of a range of local cohomology groups. In Section~\ref{sec:module-structure} we recall the quiver description of the category $\opmod_{\GL}(\D_X)$ and use it in conjunction with the vanishing results of the earlier sections to provide an inductive proof of Theorem~\ref{thm:2}. We also derive Theorem~\ref{thm:lyub-square} as a quick corollary of the previous local cohomology calculations.

\section{Preliminaries}\label{sec:prelim}

\subsection{Dominant weights and Schur functors}

We write $\bb{Z}^n_{dom}$ for the set of \defi{dominant weights} in $\bb{Z}^n$, i.e. tuples $\ll=(\ll_1,\cdots,\ll_n)\in\bb{Z}^n$ with $\ll_1\geq\ll_2\geq\cdots\geq\ll_n$. When each $\ll_i\geq 0$ we identify $\ll$ with a \defi{partition} with (at most) $n$ parts, and write $\ll\in\bb{N}^n_{dom}$. When $\ll\in\bb{Z}^n$ is not dominant, it must contain \defi{inversions}, i.e. pairs $(i,j)$ with $i<j$ and $\ll_i<\ll_j$. The \defi{size} of $\ll$ is $|\ll| = \ll_1+\cdots+\ll_n$. We sometimes use greek letters to denote weights $\ll\in\bb{Z}^n_{dom}$ and underlined roman letters to denote partitions $\x\in\bb{N}^n_{dom}$. We write $\x'$ for the \defi{conjugate} partition of $\x$, where $\x'_i$ counts the number of parts $x_j$ with $x_j\geq i$. We partially order $\bb{Z}^n_{dom}$ (and $\bb{N}^n_{dom}$) by declaring $\ll\geq \mu$ if $\ll_i\geq\mu_i$ for all $i=1,\cdots,n$. If $a\geq 0$ then we write $a\times b$ or $(b^a)$ for the sequence $(b,b,\cdots,b)$ where $b$ is repeated $a$ times.

If $V$ is a vector space with $\dim(V)=n$ and $\ll\in\bb{Z}^n_{dom}$ we write $\SS_{\ll}V$ for the corresponding irreducible representation of $\GL(V)$ (or \defi{Schur functor}). Our conventions are such that if $\ll=(d,0,\cdots,0)$ then $\SS_{\ll}V = \Sym^d V$, and if $\ll=(1^n)$ then $\bb{S}_{\ll}V=\bw^n V$. More generally, one can define $\SS_{\ll}\mc{E}$ for any locally free sheaf $\mc{E}$ of rank $n$ on some algebraic variety $X$. We write $\det(\mc{E})$ for $\bw^n\mc{E}$ and call it the \defi{determinant} of $\mc{E}$. For $m>n$ we will always think of $\bb{N}^n_{dom}$ as a subset of $\bb{N}^m_{dom}$ by identifying $\x\in\bb{N}^n_{dom}$ with $(\x,0^{m-n})$, and in this way $\SS_{\x}V$ (resp. $S_{\x}\mc{E}$) is defined whenever $\dim(V)\geq n$ (resp. $\rk(\mc{E})\geq n$).

\subsection{Gaussian binomial coefficients}\label{subsec:binomials}

For $a\geq b\geq 0$ we define the \defi{Gaussian (or $q$-)binomial coefficient} ${a\choose b}_q$ to be the polynomial in $\bb{Z}[q]$ defined by
\[{a\choose b}_q = \frac{(1-q^a)\cdot(1-q^{a-1})\cdots (1-q^{a-b+1})}{(1-q^b)\cdot(1-q^{b-1})\cdots (1-q)}.\]
These polynomials are generalizations of the usual binomial coefficients, satisfying the relations
\begin{equation}\label{eq:qbin-is-bin}
{a\choose b}_q = {a\choose a-b}_q,\quad{a\choose a}_q={a\choose 0}_q=1,\mbox{ and }{a\choose b}_1 = {a\choose b}.
\end{equation}
One significance of the $q$-binomial coefficients is that ${a\choose b}_{q^2}$ describes the Poincar\'e polynomial of the Grassmannian of $b$-dimensional subspaces of $\bb{C}^a$. As such, the coefficient of $q^j$ in ${a\choose b}_q$ computes the number of Schubert classes of (co)dimension $j$, or equivalently the number of partitions $\x$ of size $j$ contained inside the rectangular partition $(a-b)\times b$. We get
\begin{equation}\label{eq:qbin-genfun-x}
{a\choose b}_q = \sum_{\x\leq (b^{a-b})} q^{|\x|}.
\end{equation}
Using the fact that the map $\x\mapsto \x^{\circ}:=(b-x_{a-b},b-x_{a-b-1},\cdots,b-x_2,b-x_1)$ defines an involution on the set of partitions $\x\leq (b^{a-b})$, satisfying $|\x^{\circ}|=b\cdot(a-b)-|\x|$, we get that
\begin{equation}\label{eq:qbin-inv}
{a\choose b}_{q^{-1}} = {a\choose b}_q \cdot q^{-b\cdot(a-b)}.
\end{equation}
The $q$-binomial coefficients also satisfy recurrence relations analogous to the Pascal identities for usual binomial coefficients, namely
\begin{equation}\label{eq:pascal}
{a\choose b}_q = q^b\cdot {a-1\choose b}_q + {a-1\choose b-1}_q.
\end{equation}

\subsection{The ring of polynomial functions on $m\times n$ matrices and its equivariant ideals}

We consider positive integers $m\geq n\geq 1$ and let $X = \bb{C}^{m\times n}$ denote the affine space of $m\times n$ complex matrices. We let $\GL = \GL_m(\bb{C}) \times \GL_n(\bb{C})$ and consider its natural action on $X$ via row and column operations. The orbits of this action are the sets $O_p$ consisting of matrices of rank $p$, for $p=0,\cdots,n$, and their orbit closures are given by
\[ \ol{O}_p = \bigcup_{i=0}^p O_i.\] 
The coordinate ring $S$ of $X$ can be identified with the polynomial ring $S = \bb{C}[x_{ij}]$, where $1\leq i\leq m$ and $1\leq j\leq n$. If we write $I_p$ for the ideal of $p\times p$ minors of the generic matrix $(x_{ij})$, then $I_p$ is the defining ideal of the closed subvariety $\ol{O}_{p-1}$ of $X$. To keep track of the equivariance it is convenient to identify the space of linear forms in $S$ with the tensor product $\bb{C}^m \oo \bb{C}^n$, which has a natural $\GL$-action. The polynomial ring~$S$ can then be thought of as the \defi{symmetric algebra} $\Sym_{\bb{C}}(\bb{C}^m\oo\bb{C}^n)=\bigoplus_{d\geq 0}\Sym^d(\bb{C}^m\oo\bb{C}^n)$, where the component indexed by $d$ corresponds to homogeneous forms of degree $d$ in the variables $x_{ij}$. The structure of $S$ as a $\GL$-representation is governed by Cauchy's formula \cite[Corollary~2.3.3]{weyman}
\begin{equation}\label{eq:cauchy-S}
S = \bigoplus_{\x \in \bb{N}^n_{\dom}} \SS_{\x} \bb{C}^m \otimes \SS_{\x} \bb{C}^n.
\end{equation}
We write $I_{\x}\subset S$ for the ideal generated by the component $\SS_{\x} \bb{C}^m \otimes \SS_{\x} \bb{C}^n$ in the above decomposition. If $\x=(1^p)$ then the ideal $I_{\x}$ coincides with the ideal $I_p$ defined earlier. As a $\GL$-representation we have
\begin{equation}\label{eq:decomp-Ix}
I_{\x} = \bigoplus_{\y\geq\x} \SS_{\y} \bb{C}^m \otimes \SS_{\y} \bb{C}^n.
\end{equation}

\subsection{Equivariant $\D$-modules and the Grothendieck group $\Gamma_{\D}$}\label{subsec:mod-GL-DX}

We write $X=\bb{C}^{m\times n}$ as in the previous section, let $\D_X$ denote the sheaf of differential operators on $X$, and let $\opmod_{\GL}(\D_X)$ denote the category of \defi{$\GL$-equivariant coherent $\D_X$-modules} (see \cite[Section~1.1]{lor-wal}). The simple objects in $\opmod_{\GL}(\D_X)$ are $D_0,\cdots,D_n$, where $D_p$ denotes the intersection homology $\D$-module corresponding to the orbit $O_p$. As a $\GL$-representation, $D_p$ decomposes as (see \cite[Theorem~6.1]{raicu-weyman}, \cite[Main Theorem(1)]{raicu-weyman-loccoh}, \cite[Theorem~5.1]{raicu-survey})
\begin{equation}\label{eq:decomp-Dp}
 D_p = \bigoplus_{\substack{\ll_p\geq p-n \\ \ll_{p+1}\leq p-m}}\bb{S}_{\ll(p)}\bb{C}^m \oo \bb{S}_{\ll}\bb{C}^n,
\end{equation}
where 
\begin{equation}\label{eq:def-ll-p}
\ll(p) = (\ll_1,\cdots,\ll_p,(p-n)^{m-n},\ll_{p+1}+(m-n),\cdots,\ll_n+(m-n)).
\end{equation}
We note that for $p=n$ the formulas in (\ref{eq:cauchy-S}) and (\ref{eq:decomp-Dp}) coincide, which is a reflection of the fact that $D_n=S$.

We write $\Gamma_{\D}$ for the \defi{Grothendieck group of $\opmod_{\GL}(\D_X)$}, and write $[M]_{\D}$ for the class in $\Gamma_{\D}$ of an equivariant $\D_X$-module $M$. We note that the group $\Gamma_{\D}$ is a free abelian group of rank $(n+1)$, with basis given by $[D_p]_{\D}$, for $p=0,\cdots,n$. An important construction of new objects in $\opmod_{\GL}(\D_X)$ comes from considering the local cohomology groups $H_{\ol{O}_t}^j(M)$ for $j\geq 0$, $0\leq t\leq n$, and $M\in\opmod_{\GL}(\D_X)$. A first approximation to the structure of these groups is given by their class in $\Gamma_{\D}$. To keep track of this information it is convenient to write $\Gamma_{\D}[q]$ for the additive group of polynomials in the variable $q$ with coefficients in $\Gamma_{\D}$, and define
\begin{equation}\label{eq:def-HDt-M}
H^{\D}_t(M;q) = \sum_{j\geq 0} [H_{\ol{O}_t}^j(M)]_{\D}\cdot q^j \in \Gamma_{\D}[q].
\end{equation}
In the case when $M=S$, the main result of \cite{raicu-weyman} (as interpreted in \cite[Main Theorem(1)]{raicu-weyman-loccoh}) yields
\begin{equation}\label{eq:HDt-S}
H^{\D}_t(S;q) = \sum_{s=0}^t [D_s]_{\D} \cdot q^{(n-t)^2+(n-s)\cdot(m-n)} \cdot {n-1-s \choose t-s}_{q^2}.
\end{equation}

We define a pairing $\scpr{\ }{\ }_{\D}:\Gamma_{\D}[q] \times \Gamma_{\D}[q] \lra \bb{Z}[q]$ given by
\[\scpr{\gamma(q)}{\gamma'(q)}_{\D} = \sum_{s=0}^n \gamma_s(q)\cdot\gamma'_s(q),\]
where $\gamma(q) = \sum_{s=0}^n [D_s]_{\D}\cdot\gamma_s(q)$ and $\gamma'(q) = \sum_{s=0}^n [D_s]_{\D}\cdot\gamma'_s(q)$. The assertion (\ref{eq:HDt-S}) is then equivalent to
\[\scpr{H^{\D}_t(S;q)}{D_s}_{\D} = q^{(n-t)^2+(n-s)\cdot(m-n)} \cdot {n-1-s \choose t-s}_{q^2}\mbox{ for }0\leq s\leq t,\mbox{ and }\scpr{H^{\D}_t(S;q)}{D_s}_{\D} = 0\mbox{ for }s>t.\]
Notice that in the formula above we have written $D_s$ instead of $[D_s]_{\D}$, to simplify the notation. We will continue to do so as long as there is no possible source of confusion.

\subsection{Admissible representations and the Grothendieck group $\Gamma_{\GL}$}\label{subsec:admissible}

We define an \defi{admissible representation} of $\GL$ to be a representation $M$ that decomposes as
\[ M = \bigoplus_{\substack{\ll\in\bb{Z}^m_{\dom} \\ \mu\in\bb{Z}^n_{\dom}}} (\bb{S}_{\ll}\bb{C}^m \oo \bb{S}_{\mu}\bb{C}^n)^{\oplus a_{\ll,\mu}}\]
for some non-negative integers $a_{\ll,\mu}$. Examples of such representations include the polynomial ring in (\ref{eq:cauchy-S}), the ideals (\ref{eq:decomp-Ix}), and the $\D_X$-modules in (\ref{eq:decomp-Dp}). More generally, if $M$ is a finitely generated $\GL$-equivariant $S$-module or $\D_X$-module then $M$ is an admissible representation.

We write $\Gamma_{\GL}$ for the \defi{Grothendieck group of admissible $\GL$-representations}, and write $[M]_{\GL}$ for the class in $\Gamma_{\GL}$ of a representation $M$, and often refer to $[M]_{\GL}$ as a \defi{character}. The admissible representations form a semi-simple category, which implies that $[M]_{\GL}$ determines $M$ up to isomorphism. We have that $\Gamma_{\GL}$ is isomorphic to the product of copies of $\bb{Z}$ indexed by $s_{\ll,\mu} = [\bb{S}_{\ll}\bb{C}^m \oo \bb{S}_{\mu}\bb{C}^n]_{\GL}$, with $\ll\in\bb{Z}^m_{\dom}$ and  $\mu\in\bb{Z}^n_{\dom}$. We define $\Gamma_{\GL}[q]$ in analogy with $\Gamma_{\D}[q]$, and express any $\gamma(q)\in\Gamma_{\GL}(q)$ as an infinite sum
\[\gamma(q)=\sum_{\ll,\mu}a_{\ll,\mu}(q)\cdot s_{\ll,\mu},\mbox{ with }a_{\ll,\mu}(q)\in\bb{Z}.\]
We consider the partially defined pairing $\scpr{\ }{\ }:\Gamma_{\GL}[q] \times \Gamma_{\GL}[q] \lra \bb{Z}[q]$
\begin{equation}\label{eq:def-scpr-GL}
\scpr{\gamma(q)}{\gamma'(q)}_{\GL} = \sum_{\ll,\mu} a_{\ll,\mu}(q) \cdot a'_{\ll,\mu}(q)
\end{equation}
whenever the sum (\ref{eq:def-scpr-GL}) involves only finitely many non-zero terms.

We have a forgetful map that associates to a module $M\in\opmod_{\GL}(\D_X)$ the underlying admissible representation. This induces a homomorphism $\Gamma_{\D}\lra\Gamma_{\GL}$ given by $[M]_{\D} \mapsto [M]_{\GL}$. It will be important to note that this homomorphism is injective, since the characters $[D_p]_{\GL}$ described by (\ref{eq:decomp-Dp}) are linearly independent. If we combine (\ref{eq:Qp=sum-Ds}) with the case $m=n$ of (\ref{eq:decomp-Dp}) (so that $\ll(s)=\ll$ for all $s$) then it follows that as a $\GL$-representation $Q_p$ decomposes as
\begin{equation}\label{eq:decomp-Qp}
 Q_p = \bigoplus_{\ll_{p+1}\leq p-n}\bb{S}_{\ll}\bb{C}^n \oo \bb{S}_{\ll}\bb{C}^n.
\end{equation}

We extend the map $\Gamma_{\D}\lra\Gamma_{\GL}$ to an injective homomorphism $\Gamma_{\D}[q] \lra \Gamma_{\GL}[q]$, and note that for instance the image of (\ref{eq:def-HDt-M}) via this homomorphism is
\begin{equation}\label{eq:def-HGLt-M}
H^{\GL}_t(M;q) = \sum_{j\geq 0} [H_{\ol{O}_t}^j(M)]_{\GL}\cdot q^j
\end{equation}
Taking $W=\bb{S}_{\ll(p)}\bb{C}^m \oo \bb{S}_{\ll}\bb{C}^n$ to be any representation that appears in (\ref{eq:decomp-Dp}) it follows that
\begin{equation}\label{eq:scpr-D-vs-GL}
\scpr{H^{\D}_t(M;q)}{D_p}_{\D} = \scpr{H^{\GL}_t(M;q)}{W}_{\GL} 
\end{equation}
for any $M\in\opmod_{\GL}(\D_X)$, which will be particularly useful for our calculations in Section~\ref{sec:loccoh-in-Gamma}. Notice again the abuse of notation where we simply write $W$ instead of $[W]_{\GL}$, since there is no possibility of confusion.

\subsection{Flag varieties, Grassmannians, and Bott's Theorem \cite[Chapters~3 and 4]{weyman}}

Consider non-negative integers $p\leq n$ and a  complex vector space $V$ with $\dim(V)=n$. We denote by $\Flag([p,n];V)$ the variety of partial flags
\[V_{\bullet}:\quad V = V_{n}\onto V_{n-1}\cdots\onto V_p\onto 0,
\]
where $V_q$ is a $q$--dimensional quotient of $V$ for each $q=p,p+1,\cdots,n$. For $q\in[p,n]$ we write $\Qp{q}{V}$ for the tautological rank $q$ quotient bundle on $\Flag([p,n];V)$ whose fiber over a point $V_{\bullet}\in \Flag([p,n];V)$ is $V_q$. We consider the natural projection maps
\begin{equation}\label{eq:defpiq}
\pi^{(p)}_V:\Flag([p,n];V)\to \Flag([p+1,n];V), 
\end{equation}
defined by forgetting $V_p$ from the flag $V_{\bullet}$. For $p\leq n-1$, this map identifies $\Flag([p,n];V)$ with the projective bundle $\bb{P}_{\Flag([p+1,n];V)}(\Qp{p+1}{V})$, which comes with a tautological surjection
\begin{equation}\label{eq:quot-p+1-p}
\Qp{p+1}{V}\onto\Qp{p}{V}.
\end{equation}
The careful reader may have noticed that we are using the same notation $\Qp{q}{V}$ for the tautological rank $q$ quotient bundle on each of the spaces $\Flag([p,n];V)$ with $p\leq q\leq n$. This should cause no confusion (but has the advantage of simplifying the notation), as the bundle $\Qp{q}{V}$ on $\Flag([p,n];V)$ is simply the pull-back along $\pi^{(p)}$ of the corresponding bundle on $\Flag([p+1,n];V)$ when $p\leq q-1$.

The kernel of (\ref{eq:quot-p+1-p}) is a line bundle which we denote $\mc{L}_{p+1}(V)$ and note that
\begin{equation}\label{eq:Lq+1(V)}
\det(\Qp{p+1}{V}) = \mc{L}_{p+1}(V) \oo \det(\Qp{p}{V}).
\end{equation}
Just as with $\Qp{q}{V}$, there is one line bundle $\mc{L}_q(V)$ on each of the spaces $\Flag([p,n];V)$ with $p\leq q-1$. When $p>0$, the Picard group of $\Flag([p,n];V)$ is free of rank $(n-p)$, with $\mu\in\bb{Z}^{n-p}$ corresponding to the line bundle
\begin{equation}\label{eq:def-Lmu}
\mc{L}^{\mu}(V) = \bigotimes_{i=1}^{n-p} \mc{L}_{p+i}(V)^{\oo \mu_i}.
\end{equation}
Note that (\ref{eq:Lq+1(V)}) can be used to prove inductively that
\begin{equation}\label{eq:L11-detQp}
\det(V) \oo \mc{O}_{\Flag([p,n];V)} = \mc{L}^{(1^{n-p})}(V) \oo \det(\Qp{p}{V}).
\end{equation}
In particular for $p=0$ (when $\Flag([p,n];V)$ is the full flag variety) we get that $\mc{L}^{(1^n)}$ is (non-equivariantly) isomorphic to the trivial line bundle, and the Picard group has rank $(n-1)$.

If we let $\bb{G}(p,V)$ denote the Grassmannian of $p$--dimensional quotients of $V$ then we have a natural map
\begin{equation}\label{eq:def-psi-V}
\psi^{(p)}_V:\Flag([p,n];V)\lra\bb{G}(p,V),\mbox{ given by } \psi^{(p)}_V(V_{\bullet})=V_p.
\end{equation}
We abuse notation once more and write $\Qp{p}{V}$ for the tautological rank $p$ quotient bundle on $\bb{G}(p,V)$, and let $\mc{R}_{n-p}(V)$ denote the tautological rank $(n-p)$ sub-bundle, whose fiber over the point corresponding to $V_p$ is the kernel of the quotient map $V\onto V_p$. The following formulation of Bott's theorem will be useful for us throughout Section~\ref{sec:vanishing-loccoh-Jzl} (see \cite[Theorem~4.1.8]{weyman}). For $m>0$ and $\gamma\in\bb{Z}^m$ we let
\begin{equation}\label{eq:Bott-tilde}
\delta^{(m)} = (m-1,m-2,\cdots,0)\mbox{ and }\tilde{\gamma} = \sort(\gamma+\delta^{(m)}) - \delta^{(m)}
\end{equation}
where $\sort(\gamma+\delta^{(m)})\in\bb{Z}^m$ is obtained by arranging the entries of $\gamma+\delta^{(m)}$ in non-increasing order.

\begin{theorem}\label{thm:bott}
 Let $\ll\in\bb{Z}^p_{dom}$, $\mu\in\bb{Z}^{n-p}$, and let $\gamma = (\ll | \mu)\in\bb{Z}^n$ be the concatenation of $\ll$ and $\mu$. We write $\bb{F} = \Flag([p,n];V)$, $\psi = \psi^{(p)}_V$, $\pi = \pi^{(p)}_V$, and let $R^t\psi_*$ (resp. $R^t\pi_*$) denote the right derived functors of $\psi_*$ (resp. $\pi_*$). Using (\ref{eq:Bott-tilde}) we have:
 \begin{enumerate}
 \item[(a)] If $\mu + \delta^{(n-p)}$ has repeated entries then $R^t\psi_*(\SS_{\ll}\Qp{p}{V} \oo \mc{L}^{\mu}(V)) = 0$ for all $t$. Otherwise, there exists a unique $l\geq 0$ (equal to the number of inversions in $\mu + \delta^{(n-p)}$) so that
 \[R^t\psi_*(\SS_{\ll}\Qp{p}{V} \oo \mc{L}^{\mu}(V)) = \begin{cases}
 \SS_{\ll}\Qp{p}{V} \oo \SS_{\tl{\mu}}\mc{R}_{n-p}(V) & \mbox{if }t=l; \\
 0 & \mbox{otherwise}.
 \end{cases}
 \]
 \item[(b)] If $\gamma + \delta^{(n)}$ has repeated entries then $H^t(\bb{F},\SS_{\ll}\Qp{p}{V} \oo \mc{L}^{\mu}(V))  = 0$ for all $t$. Otherwise, there exists a unique $l\geq 0$ (equal to the number of inversions in $\gamma + \delta^{(n)}$) so that
 \[H^t(\bb{F},\SS_{\ll}\Qp{p}{V} \oo \mc{L}^{\mu}(V)) = \begin{cases}
 \SS_{\tl{\gamma}}V & \mbox{if }t=l; \\
 0 & \mbox{otherwise}.
 \end{cases}
 \]
 \item[(c)] If $\ll_p\geq\mu_1$ and if we let $\ll^{+} = (\ll_1,\cdots,\ll_p,\mu_1)\in\bb{Z}^{p+1}_{\dom}$ and $\mu^{-}=(\mu_2,\cdots,\mu_{n-p})\in\bb{Z}^{n-p-1}$ then
\[R^t\pi_*(\SS_{\ll}\Qp{p}{V} \oo \mc{L}^{\mu}(V)) = \begin{cases}
 \SS_{\ll^{+}}\Qp{p+1}{V} \oo \mc{L}^{\mu^{-}}(V) & \mbox{if }t=0; \\
 0 & \mbox{otherwise}.
 \end{cases}
 \]
 \end{enumerate}
\end{theorem}

\subsection{The relative setting}\label{subsec:relative}

It will sometimes be convenient to work with spaces of matrices relative to some base as follows. We let $B$ denote an algebraic variety over $\Spec(\bb{C})$ and let $\mc{F},\mc{G}$ be locally free sheaves on $B$ of ranks $m$ and $n$ respectively. We can form
\[\mc{S} = \Sym_{\mc{O}_B}(\mc{F} \oo_{\mc{O}_B}\mc{G})\]
and define $\mf{X} = \ul{\Spec}_B(\mc{S})$. We identify freely quasi-coherent $\mc{O}_{\mf{X}}$-modules $\mc{M}$ with quasi-coherent sheaves of $\mc{S}$-modules on $B$. We simply refer to such an $\mc{M}$ as an $\mc{S}$-module, and when $\mc{M}\subseteq\mc{O}_{\mf{X}}$ is an ideal sheaf, we call $\mc{M}$ an ideal in $\mc{S}$. An example of such ideal is the one defining locally matrices of rank less than $p$: we denote by $\mc{I}_p\subset\mc{S}$ the ideal generated by the subsheaf $\bw^p\mc{F}\oo\bw^p\mc{G}\subset\Sym^p(\mc{F}\oo\mc{G})\subset\mc{S}$. If we let $Z_p\subset\mf{X}$ denote the subvariety cut out by $\mc{I}_{p+1}$ then we obtain a decomposition of the local cohomology groups as $\mc{O}_B$-modules of the form
\[\mc{H}^j_{Z_p}(\mf{X},\mc{O}_{\mf{X}}) = \bigoplus_{\ll,\mu} (\SS_{\ll}\mc{F}\oo\SS_{\mu}\mc{G})^{\oplus a_{\ll,\mu}}\]
where the multiplicities $a_{\ll,\mu}$ are as computed in the case when $B=\Spec(\bb{C})$, $\mf{X}=X$, and $Z_p=\ol{O}_p$.

\subsection{The inductive structure}\label{subsec:inductive}

We let $X=\bb{C}^{m\times n}$ and consider the basic open affine $X_1\subset X$ consisting of matrices with $x_{11}\neq 0$, whose coordinate ring is the localization $S_{x_{11}}$. We let $X'=\bb{C}^{(m-1)\times(n-1)}$, and identify its coordinate ring with $S'=\bb{C}[x'_{ij}]$, with $2\leq i,j\leq n$. We have an isomorphism (given by performing row and column operations in order to eliminate entries on the first row and first column of the generic matrix)
\[X_1 \simeq X' \times \bb{C}^{m-1} \times \bb{C}^{n-1} \times \bb{C}^*\]
where the coordinate functions on $\bb{C}^{m-1}$ are $x_{i1}$, $2\leq i\leq m$, those on $\bb{C}^{n-1}$ are $x_{1j}$, $2\leq j\leq n$, the coordinate function on $\bb{C}^*$ is $x_{11}$, and 
\[x'_{ij} = x_{ij} - \frac{x_{i1}\cdot x_{1j}}{x_{11}}.\]
If we let $\pi:X_1\lra X'$ denote the projection map, and let $O'_p$ denote the orbit of rank $p$ matrices in $X'$ then
\[ \pi^{-1}(O'_p) = O_{p+1}\cap X_1\mbox{ for all }p=0,\cdots,n-1.\]
It follows that if we let $D'_p$ denote the intersection homology $\D_{X'}$-module associated with $O'_p$ then
\[\pi^*(D'_p) = (D_{p+1})_{|_{X_1}} = (D_{p+1})_{x_{11}}\mbox{ for all }p=0,\cdots,n-1.\]
If $m=n$ and if we let $\det'=\det(x'_{ij})$ then $\det = x_{11}\cdot\det'$, so $\pi^*(S'_{\det'}) = (S_{\det})_{|_{X_1}}=S_{\det\cdot\,x_{11}}$. More generally, if we define the $\D_{X'}$-modules $Q'_p$ in analogy with (\ref{eq:def-Qp}) then we obtain
\begin{equation}\label{eq:pi*Qp}
\pi^*(Q'_p) = (Q_{p+1})_{|_{X_1}} = (Q_{p+1})_{x_{11}}\mbox{ for all }p=0,\cdots,n-1.
\end{equation}
For every $S'$-module (resp. $\D_{X'}$-module) $M'$ and every closed subset $Z'\subset X'$, if we let $Z=\pi^{-1}(Z')$ and $M=\pi^*(M')$ then we have isomorphisms of $S_{x_{11}}$-modules (resp. of $\D_{X_1}$-modules)
\[\pi^*(H^j_{Z'}(M')) = H^j_Z(M)\mbox{ for all }j\geq 0.\]
In particular, we obtain
\begin{equation}\label{eq:pi*loccoh}
\pi^*(H^j_{\ol{O}'_p}(S')) = H^j_{\ol{O}_{p+1}\cap X_1}(S_{x_{11}}) = \left(H^j_{\ol{O}_{p+1}}(S)\right)_{|_{X_1}}\mbox{ for all }p=0,\cdots,n-1,\mbox{ and }j\geq 0.
\end{equation}

\section{Grothendieck group calculation of the local cohomology of simple $\D$-modules}\label{sec:loccoh-in-Gamma}

Recall that $\Gamma_{\D}$ denotes the Grothendieck group of $\opmod_{\GL}(\D_X)$, and that if $M\in\opmod_{\GL}(\D_X)$ then $[M]_{\D}$ denotes its class in $\Gamma_{\D}$. 
The main result of this section describes the class in $\Gamma_{\D}$ of the local cohomology groups with determinantal support for the modules $D_p$. 

\begin{theorem}\label{thm:loccoh-Dp}
For every $0\leq t<p\leq n\leq m$ we have the following equality in $\Gamma_{\D}[q]$:
\[
H_t^{\D}(D_p;q) = \sum_{s=0}^t [D_s]_{\D} \cdot q^{(p-t)^2+(p-s)\cdot(m-n)} \cdot {n-s \choose p-s}_{q^2} \cdot {p-1-s \choose t-s}_{q^2}
\]
\end{theorem}

We record here a special case of Theorem~\ref{thm:loccoh-Dp}, which will be used in Section~\ref{subsec:bordercase}. If $m=n=p$ and $c_t=(n-t)^2$ is the codimension of the orbit $O_t$ inside $\bb{C}^{n\times n}$ then
\begin{equation}\label{eq:H-codim-S}
 [H^{c_t}_{\ol{O}_t}(S)]_{\D} = [D_0]_{\D} + [D_1]_{\D} + \cdots + [D_t]_{\D}.
\end{equation}

\subsection{A relation between rectangular ideals and simple equivariant $\D$-modules}

We use the notational conventions from Section~\ref{subsec:admissible}. For positive integers $a,d$ and partitions $\a=(\a_1\geq\a_2\geq\cdots\geq\a_a)$ and $\b=(\b_1\geq\b_2\geq\cdots\geq\b_{m-a})$ we let
\[\ll(a,d;\a,\b) = (d+\a_1,d+\a_2,\cdots,d+\a_a,\b_1,\b_2,\cdots,\b_{m-a})\]
and consider the polynomial $h_{a\times d}(q) \in \Gamma_{\GL}[q]$ given by
\[h_{a\times d}(q) = \sum_{\a,\b} [\bb{S}_{\ll(a,d;\a,\b)}\bb{C}^m \oo \bb{S}_{\ll(a,d;\b',\a')}\bb{C}^n]_{\GL} \cdot q^{|\a|+|\b|},\]
where the sum is over partitions $\a,\b$ satisfying
\begin{equation}\label{eq:restr-alpha-beta}
 \a_1\leq n-a,\quad\a_1',\b_1\leq\min(a,d)\quad\mbox{ and }\quad\b_1'\leq m-a.
\end{equation}
The significance of the polynomials $h_{a\times d}(q)$ is that they describe the $\GL$-equivariant Hilbert series of certain simple modules over the general linear Lie superalgebra $\gl(m|n)$. As such, they provide the building blocks of the minimal free resolution over the polynomial ring $S$ of the ideals $I_{a\times d}$ (see \cite[Theorem~3.1]{raicu-weyman-syzygies} or \cite[Theorem~6.1]{raicu-survey}), namely we have
\begin{equation}\label{eq:syzygies-Iaxd}
\sum_{j\geq 0} [\Tor_j^S(I_{a\times d},\bb{C})]_{\GL}\cdot q^j = \sum_{r=0}^{n-a} h_{(a+r)\times(d+r)}(q)\cdot q^{r^2+2r}\cdot{r+\min(a,d)-1 \choose r}_{q^2}
\end{equation}
which will be used in Section~\ref{subsec:proof-loccoh-Dp} below. For now, we prove the following. 
\begin{lemma}\label{lem:scpr-W-h}
 If we let $V = \bb{S}_{(n^m)}\bb{C}^m \oo \bb{S}_{(m^n)}\bb{C}^n = \det(\bb{C}^m \oo \bb{C}^n)$ and let $d\gg 0$ then 
 \[\scpr{V\oo D_p}{ h_{a\times d}(q)}_{\GL}=0\mbox{ for }a\neq p\mbox{ and }\scpr{V\oo D_p}{ h_{p\times d}(q)}_{\GL} = q^{p\cdot(m-n)}\cdot {n\choose p}_{q^2}.\]
\end{lemma}

\begin{proof} To compute $\scpr{V\oo D_p}{ h_{a\times d}(q)}_{\GL}$, we need to characterize the partitions $\a,\b$ satisfying (\ref{eq:restr-alpha-beta}) and for which $\bb{S}_{\ll(a,d;\a,\b)}\bb{C}^m \oo \bb{S}_{\ll(a,d;\b',\a')}\bb{C}^n$ appears as a subrepresentation of $V\oo D_p$, i.e. those for which there exists a dominant weight $\mu\in\bb{Z}^n$ with $\mu_p\geq p-n$, $\mu_{p+1}\leq p-m$ (see (\ref{eq:decomp-Dp})), and such that
\begin{equation}\label{eq:mu-vs-lambda}
\mu(p) + (n^m) = \ll(a,d;\a,\b) \quad \mbox{ and } \quad \mu + (m^n) = \ll(a,d;\b',\a').
\end{equation}
If $p<a$ then it follows from (\ref{eq:def-ll-p}) that
\[p = (p-n) + n \geq \mu(p)_{p+1} + n = \ll(a,d;\a,\b)_{p+1} = d+\a_{p+1}\]
which is in contradiction with the fact that $d\gg 0$. If $p>a$ then
\[ a\geq \b_1 = \ll(a,d;\a,\b)_{a+1} = \mu(p)_{a+1} + n = \mu_{a+1} + n\geq \mu_p + n\geq (p-n) + n = p\]
which is again a contradiction. It follows that $\scpr{V\oo D_p}{ h_{a\times d}(q)}=0$ for $a\neq p$, and it remains to analyze the case $p=a$. The conditions (\ref{eq:mu-vs-lambda}) imply that
\[\mu_i = d + \a_i - n \mbox{ and }\a_i + (m-n) = \b'_i \mbox{ for all }i=1,\cdots,p.\]
Since $\b_1\leq\min(p,d)=p$ it follows from the above that $\b$ is completely determined by $\a$ via the relation $\b' = \a + ((m-n)^p)$, which in turn implies that $\b = (p^{m-n}|\a')$ and in particular
\[\b_1=\cdots=\b_{m-n}=p.\]

Suppose now that $\a$ is any partition with at most $p$ parts (i.e. $\a'_1\leq p$) and that $\a_1\leq n-p$. If we define $\b=(p^{m-n}|\a')$ then $\b_1\leq p$ and $\b'_1=\a_1+m-n\leq m-p$, so the conditions (\ref{eq:restr-alpha-beta}) hold for $a=p$, since $d\gg 0$. We next let
\[\mu_i = d + \a_i - n\mbox{ for }i=1,\cdots,p,\mbox{ and }\mu_j = \a'_{j-p} - m\mbox{ for }j=p+1,\cdots,n,\]
and observe that $\mu_p\geq p-n$ since $d\gg 0$, and that $\mu_{p+1} = \a'_1-m\leq p-m$, so $\bb{S}_{\mu(p)}\bb{C}^m \oo \bb{S}_{\mu}\bb{C}^n$ appears as a subrepresentation of $D_p$. Once we verify (\ref{eq:mu-vs-lambda}) it follows that the pair of partitions $(\a,\b)$ contributes the term $q^{|\a|+|\b|} = q^{2\cdot |\a| + p\cdot(m-n)}$ to $\scpr{V\oo D_p}{ h_{p\times d}(q)}$, hence
\[\scpr{V\oo D_p}{ h_{p\times d}(q)} = \sum_{\a} q^{2\cdot |\a| + p\cdot(m-n)} \overset{(\ref{eq:qbin-genfun-x})}{=} q^{p\cdot(m-n)}\cdot {n\choose p}_{q^2},\]
as desired. For $1\leq i\leq p$ we have that
\[\mu(p)_i + n = d + \a_i = \ll(p,d;\a,\b)_i,\mbox{ and }\mu_i + m = d + \a_i + m - n = d + \b'_i = \ll(p,d;\b',\a')_i.\]
We have moreover that for $1\leq j\leq m-n$
\[\mu(p)_{p+j} + n = (p-n) + n = p = \b_j = \ll(p,d;\a,\b)_{p+j}\]
and that for $p+1\leq j\leq n$
\[\mu_j + m = \a'_{j-p}=\ll(p,d;\b',\a')_j\mbox{ and }\]
\[\mu(p)_{m-n+j} + n = \mu_j + m = \a'_{j-p} = \b_{m-n+j-p}=\ll(p,d;\a,\b)_{m-n+j},\]
showing that (\ref{eq:mu-vs-lambda}) holds for $a=p$ and concluding our proof.
\end{proof}

\subsection{A recursive formula for Euler characteristics} We use the notational conventions from Sections~\ref{subsec:mod-GL-DX} and~\ref{subsec:admissible}, and define the \defi{Euler characteristic maps}
\[\chi:\Gamma_{\D}[q] \lra \Gamma_{\D}\mbox{ and }\chi_s:\Gamma_{\D}[q] \lra \bb{Z}\mbox{ for }s=0,\cdots,n,\]
as follows: if $\gamma(q)\in\Gamma_{\D}[q]$ is expressed as $\gamma(q) = \sum_{s=0}^n [D_s]_{\D}\cdot\gamma_s(q)$ with $\gamma_s(q)\in\bb{Z}[q]$ then we let
\begin{equation}\label{eq:def-chi}
\chi(\gamma(q)) = \gamma(-1)\mbox{ and }\chi_s(\gamma(q)) = \gamma_s(-1).
\end{equation}
Using (\ref{eq:HDt-S}) and (\ref{eq:qbin-is-bin}) we get that
\begin{equation}\label{eq:chis-HtS}
\chi_s\left(H^{\D}_t(S;q)\right) =
\begin{cases}
 (-1)^{(n-t) + (n-s)\cdot(m-n)}\cdot {n-1-s \choose t-s} & \mbox{for }s=0,\cdots,t, \\
 0 & \mbox{for }s>t. \\
\end{cases}
\end{equation}

\begin{lemma}\label{lem:recursive-chi0-Dp}
 For $t<p$ the Euler characteristics $\chi_0(H^{\D}_t(D_p;q))$ satisfy the following recurrence relation:
\begin{equation}\label{eq:recursive-chi0-Dp}
 \sum_{s=t+1}^p \chi_0(H^{\D}_t(D_s;q)) \cdot (-1)^{s\cdot(m-n)}\cdot{n-1-s\choose p-s} = (-1)^{p-t}\cdot {n-1 \choose t} - {n-1 \choose p}.
\end{equation}
\end{lemma}
\begin{proof}
The existence of a spectral sequence
\[ E_2^{i,j} = H^i_{\ol{O}_t}(H^j_{\ol{O}_p}(S)) \Longrightarrow H^{i+j}_{\ol{O}_t}(S)\]
and the fact that Euler characteristic is invariant under taking homology, imply the equality
\[\sum_{s=0}^p \chi_0(H^{\D}_t(D_s;q)) \cdot \chi_s(H^{\D}_p(S;q)) = \chi_0(H^{\D}_t(S;q))\]
which in view of (\ref{eq:chis-HtS}) can be reformulated as
\[
 \sum_{s=0}^p \chi_0(H^{\D}_t(D_s;q)) \cdot (-1)^{(n-p)+(n-s)\cdot(m-n)}\cdot{n-1-s\choose p-s} = (-1)^{(n-t) + n\cdot(m-n)}\cdot {n-1 \choose t}.
 \]
Dividing both sides by $(-1)^{(n-p)+n\cdot(m-n)}$ and moving the term $s=0$ to the right hand side yields
\begin{equation}\label{eq:recursive-2}
 \sum_{s=1}^p \chi_0(H^{\D}_t(D_s;q)) \cdot (-1)^{s\cdot(m-n)}\cdot{n-1-s\choose p-s} = (-1)^{p-t}\cdot {n-1 \choose t} - \chi_0(H^{\D}_t(D_0;q))\cdot {n-1 \choose p}.
\end{equation}
Note that for $s\leq t$ we have that the support of $D_s$ is contained in $\ol{O}_t$ and in particular $H^0_{\ol{O}_t}(D_s) = D_s$ and $H^j_{\ol{O}_t}(D_s) = 0$ for $j>0$. It follows that $\chi_0(H^{\D}_t(D_0;q)) = 1$ and $\chi_0(H^{\D}_t(D_s;q))=0$ for $0<s\leq t$, so (\ref{eq:recursive-2}) is equivalent to the desired relation (\ref{eq:recursive-chi0-Dp}).
\end{proof}

\subsection{A binomial identity}

The goal of this section is to use the recurrence relation from Lemma~\ref{lem:recursive-chi0-Dp} in order to deduce a closed formula for the Euler characteristic $\chi_0(H^{\D}_t(D_p;q))$. We prove the following.

\begin{proposition}\label{prop:chi0-Ht-Dp}
 For $0\leq t<p \leq n$ we have that
 \[\chi_0(H^{\D}_t(D_p;q)) = (-1)^{(p-t) + p\cdot(m-n)} \cdot {n\choose p} \cdot {p-1\choose t}.\]
\end{proposition}

\begin{proof} It suffices to check that the right hand side of the above equality satisfies the recursion in Lemma~\ref{lem:recursive-chi0-Dp}, that is (after cancelling some signs)
\begin{equation}\label{eq:binomial-identity}
\sum_{s=t+1}^p (-1)^{(s-t)} \cdot {n\choose s} \cdot {s-1\choose t} \cdot{n-1-s\choose p-s} = (-1)^{p-t}\cdot {n-1 \choose t} - {n-1 \choose p}.
\end{equation}
It suffices to prove that the (bivariate) generating functions of the two sides coincide, so we multiply each side by $x^t\cdot y^p$ and sum over all pairs $0\leq t<p$ of non-negative integers. We have
\[
\sum_{0\leq t<p}\left(\sum_{s=t+1}^p (-1)^{(s-t)} \cdot {n\choose s} \cdot {s-1\choose t} \cdot{n-1-s\choose p-s}\right)\cdot x^t\cdot y^p = 
\]
\[
=\sum_{s\geq 1} {n\choose s}\cdot(-y)^s\cdot\left(\sum_{t=0}^{s-1}{s-1\choose t}\cdot(-x)^t\right)\cdot\left(\sum_{p\geq s}{n-1-s\choose p-s}\cdot y^{p-s}\right)=
\]
\[
=\sum_{s\geq 1} {n\choose s}\cdot(-y)^s\cdot(1-x)^{s-1}\cdot(1+y)^{n-1-s} = \frac{(1+y)^{n-1}}{1-x}\cdot \left[\sum_{s\geq 1} {n\choose s} \cdot \left(\frac{-y\cdot(1-x)}{1+y}\right)^s\right]=
\]
\begin{equation}\label{eq:genfun-LHS}
= \frac{(1+y)^{n-1}}{1-x}\cdot\left[\left(1-\frac{y\cdot(1-x)}{1+y}\right)^n - 1\right] = \frac{(1+x y)^n}{(1-x)\cdot(1+y)} - \frac{(1+y)^{n-1}}{1-x} .
\end{equation}
We split the generating function of the right hand side of (\ref{eq:binomial-identity}) into two parts, as follows.
\begin{equation}\label{eq:genfun-RHS-1}
\sum_{0\leq t<p}(-1)^{p-t}\cdot {n-1 \choose t}\cdot x^t\cdot y^p = \sum_{t\geq 0}{n-1 \choose t}\cdot (xy)^t \cdot \left(\sum_{p>t} (-y)^{p-t} \right) = (1+x y)^{n-1} \cdot \left(\frac{-y}{1+y}\right),
\end{equation}
and
\begin{equation}\label{eq:genfun-RHS-2}
\sum_{0\leq t<p}{n-1 \choose p}\cdot x^t\cdot y^p = \sum_{p\geq 0}{n-1 \choose p}\cdot \frac{1-x^p}{1-x}\cdot y^p = \frac{1}{1-x}\cdot\left((1+y)^{n-1} - (1+xy)^{n-1}\right).
\end{equation}
Taking the difference between (\ref{eq:genfun-RHS-1}) and (\ref{eq:genfun-RHS-2}) we obtain
\[(1+x y)^{n-1} \cdot \left(\frac{1}{1-x} - \frac{y}{1+y}\right) - \frac{(1+y)^{n-1}}{1-x} = \frac{(1+xy)^n}{(1-x)\cdot(1+y)} - \frac{(1+y)^{n-1}}{1-x}\]
which is the same as (\ref{eq:genfun-LHS}), proving the identity (\ref{eq:binomial-identity}). 
\end{proof}

\subsection{The proof of Theorem~\ref{thm:loccoh-Dp}}\label{subsec:proof-loccoh-Dp}

The conclusion of Theorem~\ref{thm:loccoh-Dp} can be rephrased using (\ref{eq:def-HDt-M}) as
\begin{equation}\label{eq:scpr-HDp-Ds}
\scpr{H^{\D}_t(D_p;q)}{D_s}_{\D} = q^{(p-t)^2+(p-s)\cdot(m-n)} \cdot {n-s \choose p-s}_{q^2} \cdot {p-1-s \choose t-s}_{q^2}\quad\mbox{ for }s=0,\cdots,t.
\end{equation}
The fact that $\scpr{H^{\D}_t(D_p;q)}{D_s}_{\D}=0$ for $s>t$ follows since we are considering local cohomology groups with support in $\ol{O}_t$, and the modules $D_s$ with $s>t$ have strictly larger support.

We note that the polynomial on the right hand side of the above formula is invariant under subtracting one from each of $m,n,p,t$ and $s$. If we restrict the local cohomology groups to the basic open affine $X_1 = (x_{11}\neq 0)$ and use the inductive structure as explained in Section~\ref{subsec:inductive} then it follows that for $s>0$
\[\scpr{H^{\D}_t(D_p;q)}{[D_s]}_{\D} = \scpr{\sum_{j\geq 0} [H^j_{\ol{O}'_{t-1}}(D'_{p-1})] \cdot q^j}{[D'_{s-1}]}_{\D}\]
so the desired conclusion follows by induction. We are left with considering the case $s=0$, where we need to verify that
\[\scpr{H^{\D}_t(D_p;q)}{[D_0]}_{\D} = q^{(p-t)^2+p\cdot(m-n)} \cdot {n \choose p}_{q^2} \cdot {p-1 \choose t}_{q^2}.\]
We consider a \defi{witness representation} for the module $D_0$ (as in (\ref{eq:scpr-D-vs-GL})) defined by
\[W = \bb{S}_{(-n^m)}\bb{C}^m \oo \bb{S}_{(-m^n)}\bb{C}^n = \det(\bb{C}^m \oo \bb{C}^n)^{\vee},\]
so that it suffices to verify that
\[\scpr{H^{\GL}_t(D_p;q)}{W}_{\GL} = q^{(p-t)^2+p\cdot(m-n)} \cdot {n \choose p}_{q^2} \cdot {p-1 \choose t}_{q^2}.\]

We prove this equality in two steps:
\begin{enumerate}
 \item\label{it:step-1} We show the inequality $\leq$, where $\sum a_i \cdot q^i\leq \sum b_i \cdot q^i$ if and only if $a_i\leq b_i$ for all $i$.
 \item\label{it:step-2} We show that after plugging in $q=-1$ we obtain an equality.
\end{enumerate}

For the inequality in (\ref{it:step-1}) we begin by recalling that $\ol{O}_t$ is defined by the ideal $I_{t+1}$ of $(t+1)\times(t+1)$ minors of the generic matrix, and that the sequence of ideals $I_{(t+1)\times d}$ is cofinal with the sequence of powers of $I_{t+1}$. It follows from \cite[Exercise~A1D.1]{eisenbud-syzygies} that
\begin{equation}\label{eq:loccohDp=varinjlim}
H^j_{\ol{O}_t}(D_p) = \varinjlim_d \Ext^j_S(S/I_{(t+1)\times d},D_p).
\end{equation}
We compute the $\Ext$ modules in the above limit from the minimal resolution of $S/I_{(t+1)\times d}$ described in \cite{raicu-weyman-syzygies}. We have that $\Ext^j_S(S/I_{(t+1)\times d},D_p)$ is the $j$-th cohomology group of a complex $F^{\bullet}$ where
\[F^j = \Tor_j^S(S/I_{(t+1)\times d},\bb{C})^{\vee} \oo_{\bb{C}} D_p.\]
Notice that $\Tor_0^S(S/I_{(t+1)\times d},\bb{C}) = \bb{C}$ so that $F^0 = D_p$ and $\scpr{F^0}{W}=0$ since $p>0$. Notice also that
\[\Tor_j^S(S/I_{(t+1)\times d},\bb{C}) = \Tor_{j-1}^S(I_{(t+1)\times d},\bb{C})\mbox{ for }j\geq 1,\]
so taking $d\gg 0$ (in particular $d\geq t+1$) we have that
\begin{equation}
\begin{aligned}\label{eq:genfun-Fj}
 \scpr{\sum_{j\geq 0}  [F^j]_{\GL}\cdot q^j}{W}_{\GL} &= \scpr{W^{\vee}\oo D_p}{\sum_{j\geq 0} [\Tor_j^S(S/I_{(t+1)\times d},\bb{C})]_{\GL}\cdot q^j}_{\GL} \\
 &= q\cdot \scpr{W^{\vee}\oo D_p}{\sum_{j\geq 0} [\Tor_j^S(I_{(t+1)\times d},\bb{C})]_{\GL}\cdot q^j}_{\GL} \\
 &= \sum_{r=0}^{n-1-t}\scpr{W^{\vee}\oo D_p}{ h_{(t+1+r)\times(d+r)}(q)}\cdot q^{r^2+2r+1}\cdot {r+t\choose t}_{q^2}
\end{aligned}
\end{equation}
where the last equality follows from (\ref{eq:syzygies-Iaxd}) by taking $a=t+1$, using the fact that $\min(t+1,d)=t+1$, and noting that ${r+t\choose r}_{q^2}={r+t\choose t}_{q^2}$. Letting $a=t+1+r$ and $V=W^{\vee}$ in Lemma~\ref{lem:scpr-W-h} it follows that the only term that survives in (\ref{eq:genfun-Fj}) is the one correponding to $r=p-t-1$, which yields
\[\scpr{\sum_{j\geq 0}  [F^j]_{\GL}\cdot q^j}{W}_{\GL} = q^{p\cdot(m-n)}\cdot {n\choose p}_{q^2} \cdot q^{(p-t)^2}\cdot {p-1\choose t}_{q^2}.\]
Since $\Ext^j_S(S/I_{(t+1)\times d},D_p)$ is obtained as the $j$-th cohomology group of $F^{\bullet}$ and since $W$ does not occur in any two consecutive terms of $F^{\bullet}$, it follows that $\scpr{\Ext^j_S(S/I_{(t+1)\times d},D_p)}{W}_{\GL}=\scpr{F^j}{W}_{\GL}$ for all $j$, and using (\ref{eq:loccohDp=varinjlim}) we conclude that
\begin{equation}\label{eq:upper-bound-Hj-Dp}
\scpr{H^{\D}_t(D_p;q)}{D_0}_{\D}=\scpr{H^{\GL}_t(D_p;q)}{W}_{\GL} \leq q^{(p-t)^2+p\cdot(m-n)} \cdot {n \choose p}_{q^2} \cdot {p-1 \choose t}_{q^2}.
\end{equation}

\begin{remark}
 It may be tempting to argue at this point that the relation above is in fact an equality, but that would require to prove that the maps in the directed system (\ref{eq:loccohDp=varinjlim}) are injective, at least when restricted to the $W$-isotypic component. When $p=n$ we have $D_p=S$ and the maps are indeed injective (see \cites{raicu-weyman,raicu-regularity}), but we do not know what happens when $p<n$.
\end{remark}

Since the exponents of $q$ appearing in (\ref{eq:upper-bound-Hj-Dp}) with non-zero coefficient have the same parity, it follows that in order to prove the equality and conclude Step~(\ref{it:step-2}) of our argument, it suffices to check that equality holds in (\ref{eq:upper-bound-Hj-Dp}) after plugging in $q=-1$. In this case the left hand side becomes $\chi_0(H^{\D}_t(D_p;q))$, while the right hand side becomes $(-1)^{(p-t)+p\cdot(m-n)} \cdot {n \choose p} \cdot {p-1 \choose t}$, so the conclusion follows from Proposition~\ref{prop:chi0-Ht-Dp}.

One consequence of (\ref{eq:upper-bound-Hj-Dp}) is a vanishing result for the local cohomology groups $H^j_{\ol{O}_t}(D_p)$, based solely on the parity of $j$. Similar vanishing results, proved using more refined techniques in Sections~\ref{sec:vanishing-loccoh-Jzl} and~\ref{sec:H1-vanishing}, will play an important role in analyzing square matrices.

\begin{corollary}\label{cor:vanishing-H-Ot-Dp}
 If $j\not\equiv (p-t) + p\cdot(m-n)\ (\opmod 2)$ then $H^j_{\ol{O}_t}(D_p)=0$. In particular, when $m=n$ we may have $H^j_{\ol{O}_t}(D_p)\neq 0$ only when $j\equiv (p-t)\ (\opmod 2)$.
\end{corollary}

\subsection{The proof of Theorem~\ref{thm:lyub-non-square}}\label{subsec:proof-lyub-non-square}

We have
\[
\begin{aligned}
L_p(q,w) &= \sum_{i,j\geq 0} \scpr{H^i_{O_0}\left(H^{mn-j}_{\ol{O}_p}(S)\right)}{D_0}_{\D}\cdot q^i\cdot w^j \\
&= \sum_{i\geq 0}\left(\sum_{s=0}^p \scpr{H^i_{O_0}(D_s)}{D_0}\cdot q^i \cdot\left(\sum_{j\geq 0}\scpr{H^{mn-j}_{\ol{O}_p}(S)}{D_s}_{\D}\cdot w^j\right)\right)\\
\end{aligned}
\]
where the first equality follows from (\ref{eq:def-lyub-nos}) and the second from the fact that $\opmod_{\GL}(\D_X)$ is semisimple, and the fact that local cohomology commutes with direct sums. We obtain by reversing the summation order that
\[
\begin{aligned}
L_p(q,w) &= \sum_{s=0}^p \scpr{H_0^{\D}(D_s;q)}{D_0}_{\D} \cdot\scpr{H_p^{\D}(S;w^{-1})}{D_s}_{\D} \cdot w^{mn} \\
&\overset{(\ref{eq:scpr-HDp-Ds}),(\ref{eq:HDt-S})}{=} \sum_{s=0}^p q^{s^2+s\cdot(m-n)}\cdot {n\choose s}_{q^2}\cdot w^{-(n-p)^2-(n-s)\cdot(m-n)}\cdot{n-1-s\choose p-s}_{w^{-2}}\cdot w^{mn}
\end{aligned}
\]
Using (\ref{eq:qbin-inv}), it follows that in order to prove (\ref{eq:lyub-non-square}) it suffices to verify the identity
\[p^2+2p+s\cdot(m+n-2p-2) = -(n-p)^2-(n-s)\cdot(m-n) - 2\cdot(p-s)\cdot(n-1-p) + mn\]
which follows by inspection after expanding the products.

\section{Vanishing of local cohomology for the subquotients $J_{\x,p}$}\label{sec:vanishing-loccoh-Jzl}

Throughout this section we let $m=n$, and in order to keep track of the two distinct copies of $\bb{C}^n$ we will denote them by $F$ and $G$ respectively. We will then let $X=(F\oo G)^{\vee}$ and $S=\Sym_{\bb{C}}(F\oo G)$ be the coordinate ring of $X$. Finally, we write $\GL=\GL(F) \times \GL(G)$. The goal of this section is to revisit the construction of a class of $\GL$-equivariant $S$-modules which have played a prominent role in describing the graded components of $\Ext$ and local cohomology modules for determinantal ideals and their thickenings \cites{raicu-weyman,raicu-regularity}, and to prove vanishing results for some of their local cohomology groups. These modules are indexed by pairs $(\x,p)$ with $\x$ a partition and $p$ a non-negative integer, and are denoted $J_{\x,p}$ (see Section~\ref{subsec:Jzls} for their construction). We write $\mf{m}$ for the maximal homogeneous ideal of the polynomial ring $S$, so that $H^j_{\mf{m}}(\bullet) = H^j_{O_0}(\bullet)$. Our key vanishing result below will be proved in Section~\ref{subsec:proof-vanishing-lc-Jzl}.

\begin{theorem}\label{thm:vanishing-lc-Jzl}
 Suppose that $0\leq p\leq n$ and that $\x\in\bb{N}^n_{\dom}$ with $x_1=\cdots=x_p$. We have
 \begin{itemize}
  \item[(a)] $H^1_{\mf{m}}(\Ext^j_S(J_{\x,p},S)) = 0$ for all $j\geq 0$.
  \item[(b)] If $0\leq t\leq p$ then $H^k_{\ol{O}_t}(J_{\x,p}) = 0\mbox{ for }k \not\equiv p-t\ (\opmod\ 2)$.
 \end{itemize}
\end{theorem}

\subsection{The $J_{\x,p}$-modules and their relative versions}\label{subsec:Jzls}

For $0\leq p\leq n$ we define
\[ X^{(p)} = \Flag([p,n];F) \times \Flag([p,n];G),\]
noting that $X^{(n)} = \Spec(\bb{C})$. On $X^{(p)}$ we have a natural sheaf of algebras given by
\[ \mc{S}^{(p)} = \Sym_{\mc{O}_{X^{(p)}}}(\Qp{p}{F} \oo \Qp{p}{G}) = \bigoplus_{\x\in\bb{N}^p_{dom}} \SS_{\x}\Qp{p}{F} \oo \SS_{\x}\Qp{p}{G},\]
where the last equality comes from Cauchy's formula just like (\ref{eq:cauchy-S}). Note that when $p=n$ we get $\mc{S}^{(n)} = S$. We define $Y^{(p)} = \ul{\Spec}_{X^{(p)}}\mc{S}^{(p)}$, which is a vector bundle over $X^{(p)}$ whose fiber can be identified locally with the space of $p\times p$ matrices (see Section~\ref{subsec:relative}). For $\x\in\bb{N}^p_{dom}$ we let $\mc{I}_{\x}^{(p)}$ denote the ideal in $\mc{S}^{(p)}$ (see also (\ref{eq:decomp-Ix})) generated by $\SS_{\x}\Qp{p}{F} \oo \SS_{\x}\Qp{p}{G}$, and define
\[\mc{I}_{\X}^{(p)} = \sum_{\x\in\X} \mc{I}_{\x}^{(p)}\mbox{ for any subset }\X\subset\bb{N}^p_{dom}.\]
We define for $l<p$ and $\z\in\bb{N}^p_{dom}$ the subset of $\bb{N}^p_{dom}$
\[\scs(\z,l;p) = \{\y\in\bb{N}^p_{dom}:\y\geq\z\mbox{ and }y_i>z_i\mbox{ for some }i>l\},\]
and consider the $\mc{S}^{(p)}$-modules defined by
\[\mc{J}_{\z,l}^{(p)} = \mc{I}_{\z}^{(p)} / \mc{I}_{\scs(\z,l;p)}^{(p)},\]
with the convention that $\scs(\z,p;p)=\emptyset$ and $\mc{J}_{\z,p}^{(p)} = \mc{I}_{\z}^{(p)}$. When $p=n$ and $\x\in\bb{N}^n_{dom}$ we have $I_{\x} = \mc{I}_{\x}^{(n)}$ as in (\ref{eq:decomp-Ix}), and we write $J_{\x,l} = \mc{J}_{\x,l}^{(n)}$. The ideals $I_{\x}$ and the $S$-modules $J_{\x,l}$ have been studied in \cite[Section~2B]{raicu-weyman} and \cite[Section~2.1]{raicu-regularity}. As noted in \cite[Lemma~3.1(a)]{raicu-weyman}, if we consider the line bundle
\begin{equation}\label{eq:defdetp}
\det^{(p)} = \det\Qp{p}{F} \oo \det\Qp{p}{G}
\end{equation}
then we have $\mc{J}_{\x,l}^{(p)} \oo \det^{(p)} = \mc{J}_{\x+(1^p),l}^{(p)}$. This allows us to define $\mc{J}_{\ll,l}^{(p)}$ for any $\ll\in\bb{Z}^p_{dom}$: if $\ll = \x - (d^p)$ for some $d\in\bb{Z}_{\geq 0}$ and $\x\in\bb{N}^p_{dom}$, we let
\begin{equation}\label{eq:def-Jpl}
\mc{J}_{\ll,l}^{(p)} = \mc{J}_{\x,l}^{(p)} \oo \left(\det^{(p)}\right)^{\oo (-d)}.
\end{equation}


For $p+1\leq q\leq n$, we consider the line bundle on $X^{(p)}$ given by (see the notation in (\ref{eq:Lq+1(V)}))
\[\mc{L}_q = \mc{L}_q(F) \oo \mc{L}_q(G)\]
and for $\mu\in\bb{Z}^{n-p}$ we define in analogy with (\ref{eq:def-Lmu}) 
\[\mc{L}^{\mu} = \bigotimes_{i=1}^{n-p} \mc{L}_{p+i}^{\oo \mu_i}.\]
For $\ll\in\bb{Z}^p_{dom}$, $l\leq p$ and $\mu\in\bb{Z}^{n-p}$ we define the $\mc{S}^{(p)}$-module (with $\mc{S}^{(p)}$-action inherited from $\mc{J}_{\ll,l}^{(p)}$)
\[\mc{M}_{\ll,l;\mu}^{(p)} = \mc{J}_{\ll,l}^{(p)} \oo_{\mc{O}_{X^{(p)}}} \mc{L}^{\mu}.\]
We note that if $\y\in\bb{N}^{n-p}_{dom}$ and $d\geq y_1$, and if we define $\x\in\bb{N}^n_{dom}$ by letting
\begin{equation}\label{eq:xfromyd}
x_1=\cdots=x_p=d\mbox{ and }x_{p+i}=y_i\mbox{ for }i=1,\cdots,n-p,
\end{equation}
then the module $\mc{M}_{(d^p),p;\y}^{(p)}$ coincides with the one denoted by $\mc{M}_{\x,p}$ in \cite[(3-8)]{raicu-weyman}. It follows from \cite[Lemma~3.2]{raicu-weyman} that if we define $\x$ as in (\ref{eq:xfromyd}) then
\begin{equation}\label{eq:H^k-Mpp}
H^k\left(X^{(p)},\mc{M}_{(d^p),p;\y}^{(p)}\right) = \begin{cases}
 J_{\x,p} & \mbox{if }k=0, \\
 0 & \mbox{otherwise.}
\end{cases}
\end{equation}
We will be interested more generally in the cohomology groups of $\mc{M}_{\ll,l;\mu}^{(p)}$ for $l\leq p$, which are naturally $S$-modules. It will be useful to note that (\ref{eq:L11-detQp}) yields $\det^{(n)} = \det^{(p)} \oo\;\mc{L}^{(1^{n-p})}$ and therefore
\begin{equation}\label{eq:Moo-detn}
\mc{M}_{\ll+(1^p),l;\mu+(1^{n-p})}^{(p)} = \mc{M}_{\ll,l;\mu}^{(p)} \oo \det^{(n)}.
\end{equation}

\begin{theorem}\label{thm:filtration-HkM}
Let $0\leq q\leq p$ and $k\geq 0$, suppose that $\ll\in\bb{Z}^p_{dom}$ with $\ll_1=\cdots=\ll_q$, and that $\mu\in\bb{Z}^{n-p}$. The cohomology group $H^k\left(X^{(p)},\mc{M}_{\ll,q;\mu}^{(p)}\right)$ admits an $S$-module composition series with composition factors isomorphic to $J_{\nu,l}$ for $l\leq q$ and $\nu\in\bb{Z}^n_{dom}$. Moreover, if $\ll_p\leq\mu_j$ for some $j$ then the composition series can be chosen in such a way that each $J_{\nu,l}$ appearing as a composition factor satisfies $\nu_1=\cdots=\nu_{l+1}$.
\end{theorem}

\begin{proof}
 Using (\ref{eq:Moo-detn}) and the fact that $\det^{(n)}$ is a trivial bundle with fiber $\det(F)\oo\det(G)$ we obtain
 \[H^k\left(X^{(p)},\mc{M}_{\ll+(1^p),q;\mu+(1^{n-p})}^{(p)}\right) = H^k\left(X^{(p)},\mc{M}_{\ll,q;\mu}^{(p)}\right) \oo (\det(F)\oo\det(G)).\]
 Since we also have that $J_{\nu+(1^n),l} = J_{\nu,l} \oo (\det(F)\oo\det(G))$, it follows that we may assume without loss of generality that $\ll\in\bb{N}^p_{dom}$ and $\mu\in\bb{N}^{n-p}$. We next reduce ourselves to the case when $\mu$ is dominant. Consider 
 \[\bb{G}^{(p)} = \bb{G}(p,F) \times \bb{G}(p,G)\]
 and the natural map $\psi^{(p)} = \psi^{(p)}_F \times \psi^{(p)}_G : X^{(p)} \lra \bb{G}^{(p)}$ (see (\ref{eq:def-psi-V})). Using Theorem~\ref{thm:bott}(a) we get that
 \[R^i\psi^{(p)}_* \left(\mc{M}_{\ll,q;\mu}^{(p)}\right) = R^{i-2l}\psi^{(p)}_* \left(\mc{M}_{\ll,q;\tl{\mu}}^{(p)}\right)\mbox{ for all }i\in\bb{Z},\]
 where $l$ is the number of inversions in $\mu+\delta^{(n-p)}$. We know moreover that $R^i\psi^{(p)}_* \left(\mc{M}_{\ll,q;\mu}^{(p)}\right)$ is non-zero for at most one value of $i$, so the Leray spectral sequence degenerates and yields
 \[
 \begin{aligned}
  H^k\left(X^{(p)},\mc{M}_{\ll,q;\mu}^{(p)}\right) &= H^{k-i}\left(\bb{G}^{(p)},R^i\psi^{(p)}_* \left(\mc{M}_{\ll,q;\mu}^{(p)}\right)\right) \\
  &= H^{k-i}\left(\bb{G}^{(p)},R^{i-2l}\psi^{(p)}_* \left(\mc{M}_{\ll,q;\tl{\mu}}^{(p)}\right)\right) = H^{k-2l}\left(X^{(p)},\mc{M}_{\ll,q;\tl{\mu}}^{(p)}\right).
 \end{aligned}
\]
Notice that if $\ll_p\leq \mu_j$ for some $j$, then (\ref{eq:Bott-tilde}) forces $\ll_p\leq\tl{\mu}_1$. With these reductions, we prove our Theorem by induction on $p$ and $q$. When $p=q=0$ we have $\mc{M}_{\ll,q;\mu}^{(p)} = \mc{L}^{\mu}$. Since $\mu$ is dominant, it follows that its higher cohomology groups vanish and
 \[H^0\left(X^{(p)},\mc{M}_{\ll,q;\mu}^{(p)}\right) = \bb{S}_{\mu}F \oo \bb{S}_{\mu}G = J_{\mu,0},\]
 proving the base case. Suppose next that $0\leq q<p$, consider the natural map (see (\ref{eq:defpiq}))
 \[\pi^{(p-1)} = \pi^{(p-1)}_F \times \pi^{(p-1)}_G: X^{(p-1)} \lra X^{(p)}\]
 and define
 \[ \ll^- = (\ll_1,\cdots,\ll_{p-1})\mbox{ and }\mu^+ = (\ll_p,\mu_1,\cdots,\mu_{n-p}).\]
 We have using Theorem~\ref{thm:bott}(c) that
 \[ R^i\pi^{(p-1)}_*\left(\mc{M}_{\ll^-,q;\mu^+}^{(p-1)}\right) = \begin{cases}
 \mc{M}_{\ll,q;\mu}^{(p)} & \mbox{if }i=0; \\
 0 & \mbox{otherwise}.
 \end{cases}
 \]
 The Leray spectral sequence degenerates again, showing that
 \[H^k\left(X^{(p)},\mc{M}_{\ll,q;\mu}^{(p)}\right) = H^k\left(X^{(p-1)},\mc{M}_{\ll^-,q;\mu^+}^{(p-1)}\right)\]
 and allowing us to obtain the desired conclusion by induction on $p$.
 
 Finally, the most interesting situation is when $p=q>0$, in which case $\ll=(d^p)$ for some $d\geq 0$. If $d>\mu_1$ then it follows from (\ref{eq:H^k-Mpp}) that $\mc{M}_{\ll,q;\mu}^{(p)}$ has no higher cohomology and
 \[H^0\left(X^{(p)},\mc{M}_{\ll,q;\mu}^{(p)}\right) = J_{\nu,p},\mbox{ where }\nu = (d^p,\mu_1,\cdots,\mu_{n-p}).\]
 Note that in this case $\nu_p \neq \nu_{p+1}$! If $d\leq\mu_1$ then we obtain a filtration of $\mc{M}=\mc{M}_{\ll,p;\mu}^{(p)}$ given by
 \[ \mc{M} = \mc{M}_0 \supset \mc{M}_1 \supset \cdots \supset \mc{M}_{\mu_1-d},\mbox{ where }\mc{M}_i = \mc{M}_{\ll+(i^p),p;\mu}^{(p)}\mbox{ for }i=0,\cdots,\mu_1-d.\]
 Since each $\mc{M}_{i+1}$ is a direct summand in $\mc{M}_i$ (as an $\mc{O}_{X^{(p)}}$-module, but not as an $\mc{S}^{(p)}$-module!) we obtain a filtration
 \begin{equation}\label{eq:filtration-H^k}
 H^k\left(X^{(p)},\mc{M}\right) \supseteq H^k\left(X^{(p)},\mc{M}_1\right) \supseteq \cdots \supseteq H^k\left(X^{(p)},\mc{M}_{\mu_1-d}\right).
 \end{equation}
 
 It follows from (\ref{eq:H^k-Mpp}) that $H^k\left(X^{(p)},\mc{M}_{\mu_1-d}\right)=0$ for $k>0$ and
 \[H^0\left(X^{(p)},\mc{M}_{\mu_1-d}\right) = J_{\nu,p}\mbox{ where }\nu_1=\cdots=\nu_{p+1}=\mu_1\mbox{ and }\nu_{p+i}=\mu_i\mbox{ for }i=2,\cdots,n-p.\]
 Moreover, since
 \[\mc{M}_i/\mc{M}_{i+1} = \mc{M}^{(p)}_{\ll^i,p-1;\mu^i},\mbox{ where }\ll^i=(d+i)^{p-1}\mbox{ and }\mu^i = (d+i,\mu_1,\cdots,\mu_{n-p}),\]
 it follows that the intermediate quotients in the filtration (\ref{eq:filtration-H^k}) have the form
 \[ H^k\left(X^{(p)},\mc{M}_i/\mc{M}_{i+1}\right) = H^k\left(X^{(p)},\mc{M}^{(p)}_{\ll^i,p-1;\mu^i}\right)\]
 which by induction (on $q$) have an $S$-module filtration with composition factors as in the statement of the Theorem. Therefore (\ref{eq:filtration-H^k}) can be further refined to obtain the desired filtration for $H^k\left(X^{(p)},\mc{M}\right)$.\end{proof}
 
We will use Theorem~\ref{thm:filtration-HkM} in conjunction with the following vanishing result. Recall that $\mf{m}$ is the maximal homogeneous ideal of the polynomial ring $S$.

\begin{lemma}\label{lem:vanishing-H1m}
 Suppose that $0\leq l\leq n$ and that $\nu\in\bb{Z}^n_{\dom}$ is such that $\nu_1=\cdots=\nu_{l}$. If $l\neq 1$ or if $l=1$ and $\nu_1=\nu_2$ then
 \[H^1_{\mf{m}}(J_{\nu,l})=0.\]
\end{lemma}

\begin{proof}
 Using graded local duality, the desired vanishing is equivalent to
 \[ \Ext^{n^2-1}_S(J_{\nu,l},S) = 0.\]
 Based on (\ref{eq:def-Jpl}), we may assume without loss of generality that $\nu\in\bb{N}^n_{\dom}$ so we can apply \cite[Theorem~3.3]{raicu-weyman} which completely describes the graded components of all the modules $\Ext^j_S(J_{\nu,l},S)$. Based on the said theorem, the vanishing of $\Ext^{n^2-1}_S(J_{\nu,l},S)$ amounts to proving that it is impossible to find integers
 \[0\leq s\leq t_1\leq\cdots\leq t_{n-l}\leq l\mbox{ and dominant weights }\a\in\bb{Z}^n_{\dom}\]
 simultaneously satisfying the following conditions:
 \[
 \begin{cases}
 l^2 + 2\sum_{j=1}^{n-l} t_j = 1 & \\
 \a_n \geq l - \nu_l - n & \\
 \a_{t_j+j} = t_j - \nu_{n+1-j} - n & \mbox{for }j=1,\cdots,n-l \\
 \a_s \geq s-n \mbox{ and }\a_{s+1}\leq s-n
 \end{cases}
 \] 
where by convention $\a_0=\infty$. The first condition already forces $l=1$ and $t_1=\cdots=t_{n-1}=0$. Applying the third condition for $j=n-1$ we obtain $\a_{n-1} = -\nu_2 - n$. Since $\a$ is dominant we must then have
\[-\nu_2 - n = \a_{n-1} \geq \a_n \geq 1 - \nu_1 - n,\]
which in turn implies $\nu_1-1\geq \nu_2$ and in particular $\nu_1\neq\nu_2$. It follows that if $l\neq 1$ or if $l=1$ and $\nu_1=\nu_2$ the above conditions cannot be satisfied and $\Ext^{n^2-1}_S(J_{\nu,l},S) = 0$, concluding our proof.
\end{proof}

\begin{remark}
 If $l=1$ and $\nu_1>\nu_2$ then $H^1_{\mf{m}}(J_{\nu,l})\neq 0$. As explained in the proof above we may assume that $\nu$ is a partition. We can then take $s=t_1=\cdots=t_{n-1}=0$ and define $\a\in\bb{Z}^n_{\dom}$ by letting
 \[\a_j = -\nu_{n+1-j} - n\mbox{ for }j=1,\cdots,n-1,\mbox{ and }\a_n = 1-\nu_1 - n.\]
 It follows that $\SS_{\a}F \oo \SS_{\a}G$ appears as a subrepresentation of $\Ext^{n^2-1}_S(J_{\nu,l},S)$, proving that $H^1_{\mf{m}}(J_{\nu,l})\neq 0$.
\end{remark}

\begin{corollary}\label{cor:vanishing-H1mHk}
 Suppose that $p,q,\ll,\mu$ are as in the statement of Theorem~\ref{thm:filtration-HkM} . If $\ll_p \leq \mu_j$ for some $j$ then 
 \[ H^1_{\mf{m}}\left(H^k\left(X^{(p)},\mc{M}_{\ll,q;\mu}^{(p)}\right)\right) = 0\mbox{ for all }k.\]
\end{corollary}

\begin{proof}
 We know by Theorem~\ref{thm:filtration-HkM} that each of the groups $H^k\left(X^{(p)},\mc{M}_{\ll,q;\mu}^{(p)}\right)$ has an $S$-module filtration with composition factors isomorphic to $J_{\nu,l}$ where $\nu_1=\cdots=\nu_{l+1}$, so it suffices to prove that $H^1_{\mf{m}}(J_{\nu,l}) = 0$ for each such factor. Since no factor has $l=1$ and $\nu_1\neq\nu_2$, the desired vanishing follows from Lemma~\ref{lem:vanishing-H1m}.
\end{proof}

We record for later use one more vanishing result which is a direct consequence of Bott's Theorem.

\begin{lemma}\label{lem:vanishingHk-odd}
 Suppose that $\mc{M}$ decomposes as an $\mc{O}_{X^{(p)}}$-module into a direct sum of sheaves of the form
\[ \mc{B} = \SS_{\nu}\Qp{q}F \oo \mc{L}^{\mu}(F)\oo \SS_{\nu}\Qp{q}G\oo \mc{L}^{\mu}(G),\]
where $\nu\in\bb{Z}^p_{\dom}$ and $\mu\in\bb{Z}^{n-p}$. We have that
\[ H^k\left(X^{(p)},\mc{M}\right) = 0\mbox{ for }k\mbox{ odd}.\]
\end{lemma}

\begin{proof}
Combining the K\"unneth Theorem with Theorem~\ref{thm:bott}(b) we see that $\mc{B}$ has non-vanishing cohomology if and only if $(\nu|\mu) + \delta^{(n)}$ has no repeated entries, in which case its only non-vanishing cohomology group is
\[H^{2l}(X^{(p)},\mc{B}) = H^{l}\left(\Flag([p,n];F),\SS_{\nu}\Qp{q}F \oo \mc{L}^{\mu}(F)\right) \oo H^{l}\left(\Flag([p,n];G),\SS_{\nu}\Qp{q}G \oo \mc{L}^{\mu}(G)\right)\]
where $l$ is the number of inversions in $(\nu|\mu)$. In particular $H^k(X^{(p)},\mc{B})=0$ for $k$ odd, so the same is true for $\mc{M}$, concluding the proof.
\end{proof}

\begin{remark}
 The above vanishing applies when $\mc{M}=\mc{M}_{\ll,q;\mu}^{(p)}$, where $0\leq q\leq p$, $\ll\in\bb{Z}^p_{\dom}$ is such that $\ll_1=\cdots=\ll_q$, and $\mu\in\bb{Z}^{n-p}$.
\end{remark}



\subsection{Proof of Theorem~\ref{thm:vanishing-lc-Jzl}}\label{subsec:proof-vanishing-lc-Jzl}

We fix $0\leq p\leq n$ and $\x\in\bb{N}^n_{\dom}$ with $x_1=\cdots=x_p$. We write $\X=X^{(p)}$, $\Y=Y^{(p)}$,  and consider the commutative diagram
\[
\xymatrix{
\Y \ar@{^{(}->}[r]^>>>>>{\iota} \ar[dr]_{\phi} & T = \Spec S \times \X \ar[d]^{\pi} \ar[r]^<<<<{\pi_T} & \X \\
& \Spec S & 
}
\]
We can identify $T$ with the total space of the trivial bundle $(F\oo G)^{\vee}$ over $\X$, and $\Y$ with a subbundle of $T$ via the inclusion $\iota$. We write $\pi_{\Y} = \pi_T \circ \iota$ for the projection map $\Y\to \X$. 

We define $\y\in\bb{N}^{n-p}_{\dom}$ by letting $y_i=x_{p+i}$ for $i=1,\cdots,n-p$, set $d=x_1$ and let $\mc{M} = \mc{M}^{(p)}_{(d^p),p;\y}$. We write $\mc{S}=\mc{S}^{(p)}$, $\mc{D}=\det^{(p)}$, and thinking of $\mc{M}$ as an $\mc{S}$-module on $\X$ we have that
\begin{equation}\label{eq:mcM-is-free}
\mc{M} = \mc{S} \oo_{\mc{O}_{\X}} \mc{V} \mbox{ where }\mc{V}=\mc{D}^{\oo d} \oo_{\mc{O}_{\X}} \mc{L}^{\y}\mbox{  is a line bundle on }\X.
\end{equation}
We can then think of $\mc{M}$ as being locally (on the base $\X$) isomorphic to $\mc{S}$, or as the invertible sheaf $\pi_\Y^*\mc{V}$ on $\Y$. The relationship between $\mc{M}$ and $J_{\x,p}$ is given by (\ref{eq:H^k-Mpp}), which can be interpreted as the equality
\begin{equation}\label{eq:Jxp=Rphi}
R\phi_*(\mc{M}) = J_{\x,p}
\end{equation}
in the derived category, where $J_{\x,p}$ is considered as a complex concentrated in cohomological degree $0$.

\begin{proof}[Proof of (a)] Observe that $\Ext^j_S(J_{\x,p},S) = R^j\Hom_S(J_{\x,p},S)$. Using (\ref{eq:Jxp=Rphi}) along with Grothendieck Duality \cite[Theorem~11.1]{hartshorne-duality} we obtain
\begin{equation}\label{eq:Gro-duality}
R\Hom_S(J_{\x,p},S) = R\Hom_S(R\phi_*(\mc{M}),S) = R\phi_*(R\ShHom_{\Y}(\mc{M},\phi^{!}S)) = R\phi_*(\mc{M}^{\vee} \oo_{\mc{O}_\Y} \phi^{!}S))
\end{equation}
where the last equality follows from the fact that $\mc{M}$ is locally free. By functoriality we have $\phi^{!}S = \iota^{!}(\pi^{!}S)$ and $\pi^{!}S = \pi_T^*\omega_{\X}[\dim \X]$, where $[-]$ indicates the shift in cohomological degree and $\omega_\X$ is the canonical bundle on $\X$ (see \cite[Section~III.2]{hartshorne-duality}). We have moreover using \cite[Section~III.6]{hartshorne-duality} that
\begin{equation}\label{eq:phishriekS}
\begin{aligned}
\phi^!S = \iota^{!}(\pi_T^*\omega_{\X}[\dim \X]) &= \iota^*R\ShHom_{T}(\iota_*\mc{O}_{\Y},\pi_T^*\omega_{\X}[\dim \X]) \\
&= \det(\mc{N}_{\Y|T})[\dim(\Y)-\dim(T)] \oo_{\mc{O}_\Y} \pi_\Y^*\omega_{\X}[\dim \X] \\
\end{aligned}
\end{equation}
where $\mc{N}_{\Y|T}$ is the normal bundle of $\Y$ in $\X$. We have $\mc{N}_{\Y|T} = \pi_\Y^*\xi^{\vee}$, where
\[ \xi = \ker\left((F\oo G)\oo \mc{O}_\X \lra \Qp{p}{F} \oo \Qp{p}{G} \right),\]
so in order to compute $\det(\mc{N}_{\Y|T})$ it suffices to compute $\det(\xi)$. We have
\[ \det(\xi) = \det(F\oo G) \oo \det(\Qp{p}{F} \oo \Qp{p}{G})^{\vee} = \left(\det^{(n)}\right)^{\oo n} \oo_{\mc{O}_\X} \left(\det^{(p)}\right)^{\oo (-p)} = \mc{D}^{\oo(n-p)} \oo_{\mc{O}_\X} \mc{L}^{(n^{n-p})},\]
where the last equality follows from the fact that $\det^{(n)} = \det^{(p)} \oo_{\mc{O}_\X} \mc{L}^{(1^{n-p})}$. We have moreover that
\[ \dim(\Y) - \dim(T) + \dim(\X) = p^2-n^2 + 2p(n-p) = - (n-p)^2,\]
and the canonical bundle on $\X$ is given by (see for instance \cite[Exercise~13,~Chapter~3]{weyman})
\[\omega_\X = \mc{D}^{\oo(p-n)} \oo_{\mc{O}_\X} \mc{L}^{(p,p+1,\cdots,p+n-1)}.\]
We can therefore rewrite (\ref{eq:phishriekS}) as
\[ \phi^!S = \pi_\Y^*\left(\mc{D}^{\oo(2p-2n)} \oo_{\mc{O}_\X} \mc{L}^{(p-n,p-n+1,\cdots,p-1)} \right)[-(n-p)^2]\]
Tensoring this with $\mc{M}^{\vee} = \pi_\Y^*\left(\mc{D}^{\oo(-d)} \oo_{\mc{O}_\X} \mc{L}^{-\y} \right)$ we obtain
\[ \mc{M}^{\vee} \oo_{\mc{O}_\Y} \phi^!S = \mc{M}^{(p)}_{\ll,p;\mu}[-(n-p)^2],\mbox{ where }\]
\[\ll=((2p-2n-d)^p)\mbox{ and }\mu_i = p-n+i-1-y_i\mbox{ for }1\leq i\leq n-p.\]
It follows from (\ref{eq:Gro-duality}) that
\[\Ext^j_S(J_{\x,p},S) = R^j\phi_*\left(\mc{M}^{(p)}_{\ll,p;\mu}[-(n-p)^2]\right) = H^{j-(n-p)^2}\left(X^{(p)},\mc{M}^{(p)}_{\ll,p;\mu}\right).\]
Since $d\geq y_1$ and $p\leq n$ it follows that $\ll_p = 2p-2n-d \leq \mu_1 = p-n-y_1$, so we can apply Corollary~\ref{cor:vanishing-H1mHk} to conclude that $H^1_{\mf{m}}(\Ext^j_S(J_{\x,p},S)) = 0$ for all $j$.
\end{proof}

\begin{proof}[Proof of (b)] We let $Z_t = \phi^{-1}(\ol{O}_t)$ and note that working relative to the base $\X$, $Z_t$ is locally the variety of $p\times p$ matrices of rank at most $t$. It is cut out inside $\Y$ by the sheaf of ideals $\mc{I}_{1\times(t+1)} \subset \mc{S}$. It follows from the discussion in Section~\ref{subsec:relative} that if we set $m=n=p$ in~(\ref{eq:HDt-S}) then
\[\mc{H}^j_{Z_t}(\Y,\mc{O}_\Y) = \mc{H}^j_{\mc{I}_{1\times(t+1)}}(\X,\mc{S}) = \begin{cases}
0 & \mbox{if }j\not\equiv (p-t)\ (\opmod\ 2), \\
\bigoplus \SS_{\ll}\Qp{p}{F} \oo \SS_{\ll}\Qp{p}{G} & \mbox{if }j\equiv (p-t)\ (\opmod\ 2),
\end{cases}
\]
where the direct sum is over some collection of weights $\ll\in\bb{Z}^p_{\dom}$ (with repetitions allowed), whose precise description follows from (\ref{eq:decomp-Dp}), but is not relevant for the rest of the argument. It follows from (\ref{eq:mcM-is-free}) that
\begin{equation}\label{eq:loccoh-mcM}
\mc{H}^j_{Z_t}(\Y,\mc{M}) = \begin{cases}
0 & \mbox{if }j\not\equiv (p-t)\ (\opmod\ 2), \\
\mc{V}\oo_{\mc{O}_\X}\left(\bigoplus \SS_{\ll}\Qp{p}{F} \oo \SS_{\ll}\Qp{p}{G}\right) & \mbox{if }j\equiv (p-t)\ (\opmod\ 2).
\end{cases}
\end{equation}

Writing $\Gamma_Z$ for the functor of cohomology with support in $Z$, we get a natural isomorphism
\[\Gamma_{\ol{O}_t} \circ \phi_* = \phi_* \circ \Gamma_{Z_t}\]
which yields in the derived category
\[ R\Gamma_{\ol{O}_t}(J_{\x,p}) = R\Gamma_{\ol{O}_t}(R\phi_*\mc{M}) = R\phi_*(R\Gamma_{Z_t}(\mc{M})).\]
This means that we have a spectral sequence
\[E_2^{i,j}=H^i(\Y,\mc{H}^j_{Z_t}(\Y,\mc{M})) \Longrightarrow H^{i+j}_{\ol{O}_t}(J_{\x,p}).\]
We have noted in (\ref{eq:loccoh-mcM}) that $\mc{H}^j_{Z_t}(\Y,\mc{M})=0$ when $j\not\equiv (p-t)\ (\opmod\ 2)$, and it follows from Lemma~\ref{lem:vanishingHk-odd} and (\ref{eq:loccoh-mcM}) that $H^i(\Y,\mc{H}^j_{Z_t}(\Y,\mc{M}))=0$ when $i$ is odd. It follows that
\[ E_2^{i,j} = 0 \mbox{ when }i+j\not\equiv (p-t)\ (\opmod\ 2),\]
proving that $H^k_{\ol{O}_t}(J_{\x,p}) = 0$ for $k\not\equiv (p-t)\ (\opmod\ 2)$, as desired.
\end{proof}

\section{More vanishing of local cohomology}\label{sec:H1-vanishing}

The goal of this short section is to prove two vanishing results, which are based on Theorem~\ref{thm:vanishing-lc-Jzl} and will constitute important ingredients in describing the module structure of local cohomology groups for square matrices. We continue to assume as in Section~\ref{sec:vanishing-loccoh-Jzl} that $m=n$.

\begin{theorem}\label{thm:H1-vanishing}
 For all $p < n$ and all $j\geq 0$ we have that
 \[ H^1_{\mf{m}}(H^j_{\ol{O}_p}(S)) = 0.\]
\end{theorem}

\begin{proof}
As in (\ref{eq:loccohDp=varinjlim}) we can write
\[H^j_{\ol{O}_p}(S) = \varinjlim_d \Ext^j_S(S/I_{(p+1)\times d},S).\]
Since local cohomology commutes with direct limits, it is sufficient to prove that
\[H^1_{\mf{m}}(\Ext^j_S(S/I_{(p+1) \times d},S)) = 0.\]
Using \cite[Lemma 2.2]{raicu-weyman} (with the notation there, we choose $\ul{x}$ to be the zero partition and $\y=(d^{p+1})$), we see that the modules $S/I_{(p+1)\times d}$ admit a finite filtration by $S$-submodules whose successive quotients are of the form $J_{\ul{z},p}$, with $z_1=\dots=z_p (=z_{p+1})$. By \cite[Corollary 3.5]{raicu-weyman} (see also \cite[Theorem~3.2]{raicu-regularity}), this induces a filtration on $\Ext^j_S(S/I_{(p+1)\times d},S)$ with successive quotients $\Ext^j_S(J_{\ul{z},p},S)$. The conclusion follows now from Theorem~\ref{thm:vanishing-lc-Jzl}(a).
\end{proof}

The following should be seen as an analogue of Corollary~\ref{cor:vanishing-H-Ot-Dp}.

\begin{theorem}\label{thm:HOt-Qp-vanishing}
 If $t\leq p$ then we have that for all $k\not\equiv p-t\ (\rm{mod}\ 2)$
 \[ H^k_{\ol{O}_t}(Q_p) = 0.\]
\end{theorem}

\begin{proof}
 Note that since $S_{\det}=Q_n$, we have by (\ref{eq:decomp-Qp}) a decomposition
\[S_{\det} = \bigoplus_{\ll\in\bb{Z}^n_{\dom}}\SS_\ll \bb{C}^n \otimes \SS_\ll \bb{C}^n,\]
analogous to (\ref{eq:cauchy-S}), with the only difference that $\ll$ is allowed to be any dominant weight, as opposed to just a partition. In analogy with $I_{\x}$, we can then define the \defi{fractional ideals} $I_{\ll}$ to be the $S$-submodules of $S_{\det}$ generated by $\SS_\ll \bb{C}^n \otimes \SS_\ll \bb{C}^n$. We have $I_{\ll} = \det^{-1}\cdot I_{\ll+(1^n)}$, and it follows from (\ref{eq:decomp-Ix}) that
\begin{equation}\label{eq:decomp-Ill}
I_{\ll} = \bigoplus_{\mu\geq\ll} \SS_{\mu} \bb{C}^n \otimes \SS_{\mu} \bb{C}^n.
\end{equation}
We can write
\[S_{\det} = \varinjlim_d (\det^{-d}\cdot S) = \varinjlim_d I_{(-d^n)}.\]
Using (\ref{eq:def-Qp}) and (\ref{eq:decomp-Qp}) it follows that
\[\langle\det^{p-n+1}\rangle_{\D} = \bigoplus_{\ll_{p+1}\geq p+1-n}\SS_\ll \bb{C}^n \otimes \SS_\ll \bb{C}^n\]
and in particular using (\ref{eq:decomp-Ill}) we get
\[I_{(-d^n)} \cap \langle\det^{p-n+1}\rangle_{\D} = I_{((p+1-n)^{p+1},(-d)^{n-p-1})}.\]
for $d\gg 0$. We can then rewrite (\ref{eq:def-Qp}) as
\[ Q_p = \varinjlim_d\frac{I_{(-d^n)}}{I_{((p+1-n)^{p+1},(-d)^{n-p-1})}} = \varinjlim_d\frac{\det^d\cdot I_{(-d^n)}}{\det^d\cdot I_{((p+1-n)^{p+1},(-d)^{n-p-1})}} = \varinjlim_d\frac{S}{I_{(p+1)\times (d+p+1-n)}}. \]

Since local cohomology commutes with direct limits, it is enough to show that 
\[H^k_{\ol{O}_t} (S/I_{(p+1)\times (p-n+d+1)})= 0\, \mbox{ for }k\not\equiv p-t\ (\rm{mod}\ 2)\mbox{ and }d\gg 0.\]
As seen in the proof of Theorem~\ref{thm:H1-vanishing}, the modules $S/I_{(p+1)\times (p-n+d+1)}$ admit a finite composition series by $S$-submodules, with composition factors of the form $J_{\ul{z},p}$, with $z_1=\dots =z_{p}$. The desired vanishing now follows from Theorem~\ref{thm:vanishing-lc-Jzl}(b).
\end{proof}

\section{Module structure of local cohomology groups}\label{sec:module-structure}

The goal of this section is to describe for $X=\bb{C}^{n\times n}$ the decomposition into a sum of indecomposable objects in $\opmod_{\GL}(\D_X)$ of the local cohomology groups~$H^{\bullet}_{\ol{O}_t}(D_p)$ and $H^{\bullet}_{\ol{O}_t}(Q_p)$. In the case of non-square matrices ($m>n$) we have noted that $\opmod_{\GL}(\D_X)$ is semi-simple, so the indecomposable objects are the simple modules $D_0,\cdots,D_n$, and the decomposition of the local cohomology groups into a sum of simple modules is already encoded by their class in the Grothendieck group described in Theorem~\ref{thm:loccoh-Dp}. We will therefore only be concerned with the case when $m=n$ for the rest of the section.

To state the main results of the section, we begin by considering the additive subcategory $\add(Q)$ of $\opmod_{\GL}(\D_X)$ formed by the $\D_X$-modules that are isomorphic to a direct sum of copies of $Q_0,Q_1,\dots, Q_n$. We let $\Psi$ denote the semigroup of isomorphism classes of objects in $\add(Q)$, where the semigroup operation is given by direct sum. We write $[M]$ for the class in $\Psi$ of a module $M\in\add(Q)$. We have a natural inclusion of $\Psi$ as a sub-semigroup of $\Gamma_{\D}$, given by $[M]\mapsto[M]_{\D}$. Our first theorem describes the local cohomology groups~$H^{\bullet}_{\ol{O}_t}(D_p)$ as elements of $\add(Q)$ as follows.

\begin{theorem}\label{thm:square-loccoh-Dp} 
For every $t,p,j$ with $0\leq t<p\leq n$ and $j\geq 0$ we have that $H^j_{\ol{O}_t}(D_p)\in\add(Q)$. Moreover,
\[
\sum_{j\geq 0} [H^j_{\ol{O}_t}(D_p)] \cdot q^j = \sum_{s=0}^t [Q_s] \cdot q^{(p-t)^2} \cdot m_s(q^2)
\]
holds in $\Psi[q]$, where $m_s(q)\in\bb{Z}[q]$ is computed by $m_t(q) = \displaystyle{n-t \choose p-t}_{q}$ and
\[m_s(q) = {n-s \choose p-s}_{q} \cdot {p-1-s \choose t-s}_{q} - {n-s-1 \choose p-s-1}_{q} \cdot {p-2-s \choose t-1-s}_{q} \quad\mbox{ for }s=0,\cdots,t-1.\]
\end{theorem}

\begin{proof}
 The main content of the theorem is the assertion that $H^j_{\ol{O}_t}(D_p)\in\add(Q)$, which will be proved in Proposition~\ref{prop:loccoh-Dpsum}. Since $\Psi$ embeds into $\Gamma_{\D}$, we can determine the polynomials $m_s(q)$ by expressing $[H^j_{\ol{O}_t}(D_p)]_{\D}$ in terms of $[Q_s]_{\D}$. Using the fact that $[D_s]_{\D} = [Q_s]_{\D} - [Q_{s-1}]_{\D}$ for $s\geq 1$ and $[D_0]_{\D} = [Q_0]_{\D}$, the desired formula for $m_s(q)$ follows from the case $m=n$ of Theorem~\ref{thm:loccoh-Dp}.
\end{proof}

In order to be able to compute iterated local cohomology groups, we need to be able to describe the local cohomology groups of the modules $Q_p$.

\begin{theorem}\label{thm:square-loccoh-Qp}
For every $t,p,j$ with $0\leq t<p\leq n$ and $j\geq 0$ we have that $H^j_{\ol{O}_t}(Q_p)\in\add(Q)$. Moreover,
\begin{equation}\label{eq:square-loccoh-Qp}
\sum_{j\geq 0} [H^j_{\ol{O}_t}(Q_p)] \cdot q^j = \sum_{s=0}^t [Q_s] \cdot q^{(p-t)^2 + 2(p-s)} \cdot {n-s-1\choose p-s}_{q^2} \cdot {p-s-1\choose p-t-1}_{q^2}\mbox{ holds in }\Psi[q].
\end{equation}
\end{theorem}

\subsection{The quiver description of $\opmod_{\GL}(\D_X)$}

We recall from \cite{lor-wal} the quiver-theoretical description of the category $\opmod_{\GL}(\D_X)$, referring the reader to \cite[Section~2.4]{binary} for a quick summary of the notation and properties of quiver representations that we will use. In particular, for a quiver representation $\mf{W}$ we write $\mf{W}_x$ for the vector space associated to a vertex $x$, and write $\mf{W}(\a)$ for the linear transformation attached to an arrow $\a$. We consider the quiver with relations pictured as
\begin{equation}\label{eq:double}
\doub_n: \,\,\,  \xymatrix{
(0) \ar@<0.5ex>[r]^{\alpha_1} & \ar@<0.5ex>[l]^{\beta_1} (1)
\ar@<0.5ex>[r]^{\alpha_2} & \ar@<0.5ex>[l]^{\beta_2} \cdots
\ar@<0.5ex>[r]^{\alpha_{n-1}} &\ar@<0.5ex>[l]^{\beta_{n-1}} (n-1)
\ar@<0.5ex>[r]^-{\alpha_n} & \ar@<0.5ex>[l]^-{\beta_n} (n)}
\end{equation}
where the relations are given by the condition that all $2$-cycles are zero (i.e. $\alpha_i \beta_i = 0 = \beta_i \alpha_i$ for all $i=1,\dots,n$). By \cite[Theorem~4.4]{lor-wal} we have an equivalence of categories
\begin{equation}\label{eq:equiv-modGL-repAA}
 \opmod_{\GL}(\D_X) \simeq \rep(\doub_n)
\end{equation}
between $\opmod_{\GL}(\D_X)$ and the category of finite-dimensional representations of $\doub_n$. For instance, under this equivalence the simple $\D_X$-module $D_p$ corresponds to the irreducible representation
\[
\mf{D}^{(p)}: \,\,\,  \xymatrix{
0 \ar@<0.5ex>[r]^{0} & \ar@<0.5ex>[l]^{0} \cdots
\ar@<0.5ex>[r]^{0} & \ar@<0.5ex>[l]^{0} 0
\ar@<0.5ex>[r]^{0} &\ar@<0.5ex>[l]^{0} \bb{C}
\ar@<0.5ex>[r]^{0} & \ar@<0.5ex>[l]^{0} 0 \ar@<0.5ex>[r]^{0} 
& \ar@<0.5ex>[l]^{0} \cdots \ar@<0.5ex>[r]^{0} &\ar@<0.5ex>[l]^{0} 0},
\]
where a one dimensional vector space $\bb{C}$ is placed at vertex $(p)$, and $0$ is placed at all the other vertices. It will be important to identify the quiver representations corresponding to the modules $Q_p$ in (\ref{eq:def-Qp}).

\begin{lemma}\label{lem:quivrep-Qp}
 For each $p=0,\cdots,n$, we consider the representation $\mf{Q}^{(p)}\in\rep(\doub_n)$ obtained by letting $\mf{Q}^{(p)}_{(i)}=\bb{C}$ for $0\leq i\leq p$, and $\mf{Q}^{(p)}_{(i)}=0$ for $i>p$, and with maps as pictured below
 \begin{equation}\label{eq:quiv-Qp}
\mf{Q}^{(p)}: \,\,\,  \xymatrix{
\bb{C} \ar@<0.5ex>[r]^{1} & \ar@<0.5ex>[l]^{0} \bb{C}
\ar@<0.5ex>[r]^{1} & \ar@<0.5ex>[l]^{0} \cdots
\ar@<0.5ex>[r]^{1} &\ar@<0.5ex>[l]^{0} \bb{C}
\ar@<0.5ex>[r]^{0} & \ar@<0.5ex>[l]^{0} 0 \ar@<0.5ex>[r]^{0} 
& \ar@<0.5ex>[l]^{0} \cdots \ar@<0.5ex>[r]^{0} &\ar@<0.5ex>[l]^{0} 0}
\end{equation}
We have that $\mf{Q}^{(p)}$ contains $\mf{D}^{(p)}$ as its unique irreducible subrepresentation, and that $\mf{Q}^{(p)}/\mf{D}^{(p)}\simeq\mf{Q}^{(p-1)}$. Moreover, the $\D_X$-module $Q_p$ corresponds via (\ref{eq:equiv-modGL-repAA}) to the representation $\mf{Q}^{(p)}$ for all $0\leq p\leq n$.
\end{lemma}

\begin{proof}
 The fact that $\mf{D}^{(p)}$ is a subrepresentation of $\mf{Q}^{(p)}$ and the identification $\mf{Q}^{(p)}/\mf{D}^{(p)}\simeq\mf{Q}^{(p-1)}$ follow from the definition of $\mf{Q}^{(p)}$. If $\mf{W}\subseteq\mf{Q}^{(p)}$ is a subrepresentation with $\mf{W}_{(i)}=\bb{C}$ for some $i<p$, then $\mf{W}_{(i+1)}$ contains the image under $\mf{Q}^{(p)}(\a_{i+1})$ of $\mf{W}_{(i)}$, that is $\mf{W}_{(i+1)}=\bb{C}$. It follows that $\mf{W}_{(j)}=\bb{C}$ for all $j=i,\cdots,p$, and in particular $\mf{W}$ contains $\mf{D}_p$ as a subrepresentation.

To prove that $Q_p$ corresponds to $\mf{Q}^{(p)}$ via (\ref{eq:equiv-modGL-repAA}) we argue by descending induction on $p$. Using (\ref{eq:filtration-Sdet}--\ref{eq:def-Qp}) we get that $Q_{p-1}\simeq Q_p/D_p$, proving the inductive step. It remains to address the base case $p=n$, when $Q_n=S_{\det}$. If we apply \cite[Lemma~2.4]{binary} with $G=\GL$, $Y=\bb{C}^{n\times n}$, $U=O=O_n$ the dense orbit of rank $n$ matrices, and $j:U\lra Y$ the natural inclusion, it follows that $S_{\det}=j_*j^*S$ is the injective envelope of $S=D_n$ in $\opmod_{\GL}(\D_X)$, so $S_{\det}$ corresponds via (\ref{eq:equiv-modGL-repAA}) to the injective envelope of $\mf{D}^{(n)}$. Using the quiver description of the injective envelope of a simple representation from \cite[(2.15)]{binary}, it follows that $\mf{Q}^{(n)}$ is the injective envelope of $\mf{D}^{(n)}$, concluding the proof.
\end{proof}

For each $p=0,\cdots,n$ we consider the full subcategory
\[\opmod_{\GL}^{\ol{O}_p}(\D_X)\]
of $\opmod_{\GL}(\D_X)$ consisting of modules with support contained in $\ol{O}_p$. This subcategory is closed under extensions and taking subquotients, and it corresponds via (\ref{eq:equiv-modGL-repAA}) to the subcategory $\rep(\doub_p)$ of $\rep(\doub_n)$, obtained by forgetting the vertices $(p+1),\cdots,(n)$ of the quiver $\doub_n$. We have the following important observation, which follows from \cite[(2.15)]{binary} and the equivalence with $\rep(\doub_p)$.

\begin{lemma}\label{lem:Qp-inj-proj}
 Inside the category $\opmod_{\GL}^{\ol{O}_p}(\D_X)$, the module $Q_p$ is the injective envelope of $D_p$ and the projective cover of $D_0$. In particular, $Q_p$ is indecomposable.
\end{lemma}

To describe local cohomology groups we will work mainly in the additive subcategory $\add(Q)$ of $\opmod_{\GL}(\D_X)$. One property that will be important for us is that $\add(Q)$ is closed under taking extensions and quotients.

\begin{lemma}\label{lem:ext-Qp}
For every $0\leq i,j\leq n$ we have that $\Ext^1_{\opmod_{\GL}(\D_X)}(Q_i,Q_j)=0$. In particular, every short exact sequence in $\opmod_{\GL}(\D_X)$
\begin{equation}\label{eq:ses-M1-N-M2}
 0 \lra M_1 \lra N \lra M_2 \lra 0,
\end{equation}
with $M_1,M_2 \in \add(Q)$ splits, and hence $N\in \add(Q)$. More generally, if $N\in\opmod_{\GL}(\D_X)$ has a composition series with composition factors in $\add(Q)$, then $N\in\add(Q)$.
\end{lemma}

\begin{proof}
 For the first assertion, we let $p=\max(i,j)$ and note that since $\opmod_{\GL}^{\ol{O}_p}(\D_X)$ is closed under taking extensions it suffices to prove that
 \[\Ext^1_{\opmod_{\GL}^{\ol{O}_p}(\D_X)}(Q_i,Q_j)=0.\]
 By Lemma~\ref{lem:Qp-inj-proj}, if $p=i$ then $Q_i$ is projective in $\opmod_{\GL}^{\ol{O}_p}(\D_X)$, while if $p=j$ then $Q_j$ is injective, so the above vanishing follows. Since $\Ext^1$ commutes with finite direct sums, it follows that $\Ext^1_{\opmod_{\GL}(\D_X)}(M_2,M_1)=0$ for $M_1,M_2\in\add(Q)$, and therefore (\ref{eq:ses-M1-N-M2}) splits. To prove the last assertion, we argue by induction on the length of the composition series. We write $N$ as an extension (\ref{eq:ses-M1-N-M2}), where $M_2\in\add(Q)$ and $M_1$ has a shorter composition series with composition factors in $\add(Q)$. By induction we have that $M_1\in\add(Q)$, hence (\ref{eq:ses-M1-N-M2}) splits and $N$ is also in $\add(Q)$.
\end{proof}

\begin{lemma}\label{lem:quot-Qp}
Any quotient of $Q_p$ in $\opmod_{\GL}(\D_X)$ is isomorphic to $Q_q$ for some $0\leq q\leq p$. More generally, if $M\in\add(Q)$ then any quotient of $M$ is also in $\add(Q)$.
\end{lemma}

\begin{proof}
We prove the first assertion by induction on $p$. By Lemma~\ref{lem:quivrep-Qp} and (\ref{eq:equiv-modGL-repAA}), $D_p$ is the unique simple submodule of $Q_p$, and therefore every proper quotient of $Q_p$ factors through $Q_p/D_p=Q_{p-1}$. By induction, every quotient of $Q_{p-1}$ is isomorphic to $Q_q$ for some $0\leq q\leq p-1$, so the same must be true about every proper quotient of $Q_p$.

For the last assertion we argue by induction on the length of $M$. We consider a quotient $\pi:M\onto P$ and write $M=Q_p\oplus N$ with $N\in\add(Q)$, and let $P' = \pi(Q_p)$. Using the previous paragraph, $P'\simeq Q_q$ for some $0\leq q\leq p$. The map $\pi$ induces a map of short exact sequences, where $P''=P/P'$,
\[
\xymatrix{
0 \ar[r] & Q_i \ar[r] \ar[d] & M \ar[r] \ar[d] & N \ar[r] \ar[d] & 0 \\
0 \ar[r] & P' \ar[r] & P \ar[r] & P'' \ar[r] & 0 \\
}
\]
where the vertical maps are surjective. Since $N$ has smaller length than $M$, it follows that $P''\in\add(Q)$, hence $P\in\add(Q)$ by Lemma~\ref{lem:ext-Qp}.
\end{proof}

\subsection{Local cohomology of the polynomial ring $S$}

The goal of this section is to prove that the local cohomology groups of $S$ are in $\add(Q)$, thus proving the case $p=n$ of Theorem~\ref{thm:square-loccoh-Dp}. Our argument will be inductive, starting with the observations in Section~\ref{subsec:inductive}. We let $X_1\subset X$ denote the basic open affine where $x_{11}\neq 0$, let $U = X\setminus\{0\}$, and let $j_1:X_1\to U$ denote the open immersion.

\begin{lemma}\label{lem:restr-to-X1}
 If $M,N\in\opmod_{\GL}(\D_U)$ are such that there exists a $\D_{X_1}$-module isomorphism $j_1^*M\simeq j_1^*N$ then $M\simeq N$.
\end{lemma}

\begin{proof}
 We let $Z=U\setminus X_1$ and consider the exact sequences
\[
\xymatrix{
0 \ar[r] & \mc{H}^0_Z(M) \ar[r] & M \ar[r]^<<<<<{\a} & j_{1*}j_1^* M \ar[r]^{\a'} \ar[d]_{\phi}^{\simeq} & \mc{H}^1_Z(M) \ar[r] & 0 \\
0 \ar[r] & \mc{H}^0_Z(N) \ar[r] & N \ar[r]^<<<<<{\b} & j_{1*}j_1^* N \ar[r]^{\b'} & \mc{H}^1_Z(N) \ar[r] & 0 \\
}
\]
where $\phi$ exists by assumption. Since $Z$ contains no invariant closed subset of $U$, it follows that no non-zero subquotient of $M$ and $N$ can have support in $Z$. It follows that $\mc{H}^0_Z(M)=\mc{H}^0_Z(N)=0$ and therefore $\a,\b$ are injective. Moreover, we have $\b'\circ\phi\circ\a=0$ and $\a'\circ\phi^{-1}\circ\b=0$, so that $\phi\circ\a$ (resp. $\phi^{-1}\circ\b$) lifts to an injective $\D_U$-module homomorphism $\phi_1$ (resp. $\phi_2$). Since $M$ and $N$ have finite length, it follows that their lengths coincide, and $\phi_1$ and $\phi_2$ must be isomorphisms.
\end{proof}

\begin{proposition}\label{prop:loccoh-dsum}
For all $t<n$ and $i\geq 0$ we have $H^i_{\ol{O}_t}(S)\in \add(Q)$.
\end{proposition}

\begin{proof}
We proceed by induction on $n$: if $t=0$ then each $H^i_{\ol{O}_t}(S)$ is a direct sum of copies of $D_0=Q_0$, so it is in $\add(Q)$. We may assume then that $n>t\geq 1$ and let $j : U \lra X$ denote the inclusion. For any $\D_X$-module $M$, we have the exact sequence
\begin{equation}\label{eq:cohexact}
0 \lra H^0_{\mf{m}}(M) \lra M \lra j_*j^* M \lra H^1_{\mf{m}}(M) \lra 0.
\end{equation}

We let $Q^0_p = j^*Q_p$ for $1\leq p \leq n$ and prove that $j_* Q^0_p = Q_p$. From Lemma~\ref{lem:quivrep-Qp} we see that $Q_p$ has no submodules supported at $O_0$, so $H^0_{\mf{m}}(Q_p)=0$. Choosing $M=Q_p$ in (\ref{eq:cohexact}) gives then the exact sequence
\[0 \lra Q_p \lra j_*Q^0_p \lra H^1_{\mf{m}}(Q_p) \lra 0.\]
Since $H^1_{\mf{m}}(Q_p)$ is a direct sum of copies of $D_0=Q_0$, it follows from Lemma \ref{lem:ext-Qp} that the above sequence splits. If we set $M= j_*Q^0_p$ in (\ref{eq:cohexact}) and note that the map $j_*Q^0_p\to j_*j^*j_*Q^0_p$ is an isomorphism, we get $H^0_{\mf{m}}(j_*Q_p^0)=0$. Since $H^1_{\mf{m}}(Q_p)$ is a summand of $j_*Q^0_p$ supported at $O_0$, this shows that $H^1_{\mf{m}}(Q_p)=0$ and $j_*Q^0_p= Q_p$.

We now claim that each $j^* H^i_{\ol{O}_t}(S)$ is a direct sum of copies of the $\D_U$-modules $Q^0_1,\dots, Q^0_n$. To prove this, it suffices by Lemma~\ref{lem:restr-to-X1} to show that an isomorphism exists after restricting to $X_1$. For that we have
\[j_1^*j^*H^i_{\ol{O}_t}(S) = \left(H^i_{\ol{O}_{t}}(S)\right)_{|_{X_1}} \overset{(\ref{eq:pi*loccoh})}{=} \pi^*(H^i_{\ol{O}'_{t-1}}(S')) \overset{(\ref{eq:pi*Qp})}{=} \bigoplus_{1\leq s\leq t} (Q_{s}^{\oplus a_s})_{|_{X_1}} = \bigoplus_{1\leq s\leq t} j_1^*(Q_{s}^0)^{\oplus a_s}\]
where the equality labelled (\ref{eq:pi*Qp}) uses also the induction hypothesis, and where the numbers $a_s$ are in $\bb{Z}_{\geq 0}$.

Since $j^* H^i_{\ol{O}_t}(S)$ is a direct sum of copies of $Q^0_1,\dots, Q^0_n$, and $j_*Q^0_p= Q_p$ for $1\leq p\leq n$, it follows now that $j_*j^* H^i_{\ol{O}_t}(S)\in \add(Q)$. Setting $M= H^i_{\ol{O}_t}(S)$ in (\ref{eq:cohexact}) we obtain using Theorem~\ref{thm:H1-vanishing} the exact sequence
\[ 0 \to H^0_{\mf{m}}(H^i_{\ol{O}_t}(S)) \lra H^i_{\ol{O}_t}(S) \lra j_*j^* H^i_{\ol{O}_t}(S) \lra 0.\]
Since $H^0_{\mf{m}}(H^i_{\ol{O}_t}(S))\in\add(Q)$, it follows from Lemma \ref{lem:ext-Qp} that $H^i_{\ol{O}_t}(S)\in\add(Q)$, concluding the proof.
\end{proof}

\subsection{The structure of the modules $H^{\bullet}_{\ol{O}_{p-1}}(Q_p)$}\label{subsec:bordercase}

In this section we prove the case $t=p-1$ of Theorem~\ref{thm:square-loccoh-Qp}.

\begin{lemma}\label{lem:loccoh-localize}
For all $j\geq 0$ and $t < n$ we have $H^j_{\ol{O}_t}(S_{\det})=0$.
\end{lemma}

\begin{proof}
Multiplication by the polynomial $\det$ induces an $S$-module isomorphism $S_{\det}\overset{\cdot\det}{\lra} S_{\det}$, which in turn gives rise to an isomorphism $H^j_{\ol{O}_t}(S_{\det})\overset{\cdot\det}{\lra} H^j_{\ol{O}_t}(S_{\det})$ for each $j\geq 0$. Since the polynomial $\det$ vanishes on $\ol{O}_t$ it follows that every element $m\in H^j_{\ol{O}_t}(S_{\det})$ is annihilated by $\det^k$ for some $k$. Since multiplication by $\det^k$ is an isomorphism, we conclude that $m=0$ and, since $m$ was arbitrary, that $H^j_{\ol{O}_t}(S_{\det})=0$.
\end{proof}

\begin{lemma}\label{lem:loccoh-qsum}
For all $p\leq n$ and $j\geq 0$, we have $H^j_{\ol{O}_{p-1}}(Q_p)\in \add(Q)$ and
\begin{equation}\label{eq:01-vanish-loccoh-Qp-p-1}
H^0_{\ol{O}_{p-1}}(Q_p) = H^1_{\ol{O}_{p-1}}(Q_p) = 0.
\end{equation}
\end{lemma}

\begin{proof}
The case $p=n$ follows from Lemma \ref{lem:loccoh-localize}, so we may assume that $p<n$. We consider the spectral sequence
\[ E_2^{i,j} = H^i_{\ol{O}_{p-1}} (H^j_{\ol{O}_p}(S)) \Longrightarrow H^{i+j}_{\ol{O}_{p-1}}(S).\]
If we let $c_p=(n-p)^2$ denote the codimension of $O_p$ in $X$, then we know that $H^j_{\ol{O}_p}(S)$ has support contained in $\ol{O}_{p-1}$ if $j\neq c_p$, and therefore $E_2^{i,j}=0$ if $i\neq 0$ and $j\neq c_p$. Moreover, combining Proposition~\ref{prop:loccoh-dsum} with~(\ref{eq:Qp=sum-Ds}) and~(\ref{eq:H-codim-S}) we see that $H^{c_p}_{\ol{O}_p}(S)\cong Q_p$, so we have
\begin{equation}\label{eq:E2-descr}
E_2^{i,c_p} = H^i_{\ol{O}_{p-1}}(Q_p)\mbox{ for }i\geq 0\mbox{ and }E_2^{0,j} = H^j_{\ol{O}_p}(S)\mbox{ for }j\neq c_p.
\end{equation}
It follows that the potentially non-zero groups~$E_2^{i,j}$ are arranged along a hook shape centered around the point $(i,j) = (0,c_p)$, and that the only potentially non-zero maps in the spectral sequence are the homomorphisms
\begin{equation}\label{eq:dr}
 E_2^{0,c_p+r-1} = E_r^{0,c_p+r-1} \overset{d_r}{\lra} E_r^{r,c_p} = E_2^{r,c_p}\mbox{ for }r\geq 2.
\end{equation}
It follows that
\[ E_{\infty}^{0,c_p+r-1} = \ker(d_r)\quad\mbox{and}\quad E_{\infty}^{r,c_p} = \coker(d_r)\quad\mbox{ for }r\geq 2.\]
Since $H^{k}_{\ol{O}_{p-1}}(S)=0$ for $k\equiv c_p\ (\opmod\ 2)$ by the case $p=m=n$ and $t=p-1$ of Corollary~\ref{cor:vanishing-H-Ot-Dp} it follows that $E^{0,j}_{\infty}=0\mbox{ when }j\equiv c_p\ (\opmod\ 2)$. Since $E^{0,j}_2=0$ for $j\not\equiv c_p\ (\opmod\ 2)$ by (\ref{eq:E2-descr}) and Corollary~\ref{cor:vanishing-H-Ot-Dp}, we conclude that $E^{0,j}_{\infty}=0\mbox{ for all }j\geq 0$, and in particular that all the maps $d_r$ in (\ref{eq:dr}) are injective. The vanishing of $E^{0,j}_{\infty}$ and the shape of the spectral sequence show that 
\begin{equation}\label{eq:Einfty-icp}
E^{i,c_p}_{\infty} = H^{i+c_p}_{\ol{O}_{p-1}}(S)\mbox{ for all }i\geq 0,
\end{equation}
and therefore we obtain short exact sequences
\[0\lra E^{0,c_p+r-1}_2 \overset{d_r}{\lra} E^{r,c_p}_2 \lra H^{r+c_p}_{\ol{O}_{p-1}}(S) \lra 0\]
Since the modules $E_2^{0,c_p+r-1}$ and $H^{r+c_p}_{\ol{O}_{p-1}}(S)$ are in $\add(Q)$ by Proposition~\ref{prop:loccoh-dsum} and~(\ref{eq:E2-descr}), it follows from Lemma \ref{lem:ext-Qp} that the same is true for $E^{r,c_p}_2$, i.e. $H^r_{\ol{O}_{p-1}}(Q_p)\in \add(Q)$ for all $r\geq 2$.

Since the maps (\ref{eq:dr}) do not involve any of the modules $E_r^{i,c_p}$ for $i=0,1$, it follows that
\[H^i_{\ol{O}_{p-1}}(Q_p) \overset{(\ref{eq:E2-descr})}{=} E_2^{i,c_p} = E_{\infty}^{i,c_p} \overset{(\ref{eq:Einfty-icp})}{=} H^{i+c_p}_{\ol{O}_{p-1}}(S)=0\mbox{ for }i=0,1,\]
where the vanishing of $H^{i+c_p}_{\ol{O}_{p-1}}(S)$ follows from the fact that 
\[c_{p-1}=(n-p+1)^2>i+c_p=i+(n-p)^2\mbox{ for }i=0,1\mbox{ and }p<n,\]
proving (\ref{eq:01-vanish-loccoh-Qp-p-1}) and concluding our proof.
\end{proof}

\subsection{Local cohomology of the simples $D_p$}

We are now ready to finalize the proof of Theorem~\ref{thm:square-loccoh-Dp}.

\begin{proposition}\label{prop:loccoh-Dpsum}
For every $j\geq 0$ and $t,p$ with $0\leq t <p \leq n$ we have $H^j_{\ol{O}_{t}}(D_p)\in \add(Q)$.
\end{proposition}

\begin{proof}
We prove the result by descending induction on the pair $t<p$. We begin with the case when $t=p-1$ and consider the short exact sequence
\[ 0\lra D_p \lra Q_p \lra Q_{p-1} \lra 0.\]
Since $H^0_{\ol{O}_{p-1}}(Q_{p-1}) = Q_{p-1}$, $H^j_{\ol{O}_{p-1}}(Q_{p-1}) = 0$ for $j>0$ by (\ref{eq:suppM-in-Ot}), and $H^j_{\ol{O}_{p-1}}(Q_p)=0$ for $j=0,1$ by (\ref{eq:01-vanish-loccoh-Qp-p-1}), we obtain by the long exact sequence in cohomology that
\begin{equation}\label{eq:QpDp}
H^0_{\ol{O}_{p-1}}(D_p) = 0,\ H^1_{\ol{O}_{p-1}}(D_p) = Q_{p-1},\mbox{ and }H^j_{\ol{O}_{p-1}}(D_p) = H^j_{\ol{O}_{p-1}}(Q_p)\mbox{ for }j\geq 2.
\end{equation}
It follows from Lemma~\ref{lem:loccoh-qsum} that $H^j_{\ol{O}_{p-1}}(D_p)\in \add(Q)$ for all $j\geq 0$. For the inductive step we consider $1\leq t<p$ and the spectral sequence
\[ E_2^{i,j} = H^i_{\ol{O}_{t-1}} (H^j_{\ol{O}_t}(D_p)) \Longrightarrow H^{i+j}_{\ol{O}_{t-1}}(D_p).\]
By induction, the modules $H^j_{\ol{O}_t}(D_p)$ belong to $\add(Q)$, and their summands are among $Q_0,\dots, Q_t$, since they have support contained in $\ol{O}_t$. Using the fact that for $s\leq t-1$ we have $H^0_{\ol{O}_{t-1}}(Q_s)=Q_s$ and $H^i_{\ol{O}_{t-1}}(Q_s)=0$, together with the fact that $H^i_{\ol{O}_{t-1}}(Q_t)\in\add(Q)$ proved in Lemma~\ref{lem:loccoh-qsum}, we conclude that each $E_2^{i,j}$ belongs to $\add(Q)$. Our final goal is to prove that $E_{\infty}^{i,j}\in\add(Q)$, since the modules $E_{\infty}^{i,j}$ constitute the composition factors of $H^{i+j}_{\ol{O}_{t-1}}(D_p)$ with respect to the filtration induced by the spectral sequence. By Lemma~\ref{lem:ext-Qp}, this implies that $H^k_{\ol{O}_{t-1}}(D_p)\in\add(Q)$ for all $k\geq 0$, concluding the inductive step.

Using Theorem~\ref{thm:loccoh-Dp} we have that $H^k_{\ol{O}_{t-1}}(D_p)=0$ for $k\equiv p-t\ (\opmod\ 2)$, so we only need to consider the modules $E^{i,j}_{\infty}$ when $i+j\not\equiv p-t\ (\opmod\ 2)$. We will prove by induction on $r\geq 2$ that $E_r^{i,j}$ is a quotient of $E_2^{i,j}$ when $i+j\not\equiv p-t\ (\opmod\ 2)$. Since $E_{r+1}=\ker(d_r)/\mbox{Im}(d_r)$, it suffices to check that the differentials
\[d_r^{i,j}:E_r^{i,j} \lra E_r^{i+r,j-r+1}\mbox{ are identically $0$ for }i+j\not\equiv p-t\ (\opmod\ 2).\]
Since $i+r\geq 2$ this is in turn is implied by the vanishing
\begin{equation}\label{eq:E2-vanishing}
E_2^{i,j}=0\mbox{ for }i\geq 2\mbox{ and }i+j\equiv p-t\ (\opmod\ 2),
\end{equation}
which we explain next. Theorem~\ref{thm:loccoh-Dp} implies that $H^j_{\ol{O}_t}(D_p)=0$ for $j\not\equiv p-t\ (\opmod\ 2)$, so we only need to prove (\ref{eq:E2-vanishing}) when $i\geq 2$ is even and $j\equiv p-t\ (\opmod\ 2)$. Since $H^i_{\ol{O}_{t-1}}(Q_s)=0$ for $i>0$ and $s\leq t-1$, and since $H^j_{\ol{O}_t}(D_p)$ is a direct sum of copies of $Q_0,\cdots,Q_t$, it suffices to check that
\[H^i_{\ol{O}_{t-1}}(Q_t)=0\mbox{ for }i\mbox{ even},\]
which follows from Theorem~\ref{thm:HOt-Qp-vanishing}, and concludes our proof.
\end{proof}

\subsection{Local cohomology of the indecomposables $Q_p$}\label{subsec:loccoh-Qp} The goal of this section is to prove Theorem~\ref{thm:square-loccoh-Qp}. 

\begin{proof}[Proof of Theorem~\ref{thm:square-loccoh-Qp}] If $p=n$ then it follows from Lemma~\ref{lem:loccoh-localize} that $H^j_{\ol{O}_t}(Q_n)=0$ for all $0\leq t<n$, which coincides with the formula (\ref{eq:square-loccoh-Qp}) since ${n-s-1\choose n-s}_{q^2}=0$ for all $s$. We may therefore assume that $t\leq n-2$, and proceed by induction on $p$, starting with the case $p=t+1$. Combining (\ref{eq:01-vanish-loccoh-Qp-p-1}) with (\ref{eq:QpDp}) and Theorem~\ref{thm:square-loccoh-Dp}, we get that $H^j_{\ol{O}_{p-1}}(Q_p) \in \add(Q)$ for all $j\geq 0$ and moreover
\[[Q_{p-1}]\cdot q + \sum_{j\geq 0} [H^j_{\ol{O}_{p-1}}(Q_p)] \cdot q^j = \sum_{j\geq 0} [H^j_{\ol{O}_{p-1}}(D_p)] \cdot q^j = \sum_{s=0}^{p-1} [Q_s] \cdot q \cdot m_s(q^2),\]
where $m_{p-1}(q)={n-p+1\choose 1}_q=1+q+q^2+\cdots+q^{n-p}$ and 
\[m_s(q) = {n-s\choose p-s}_q - {n-s-1\choose p-s-1}_q \overset{(\ref{eq:pascal})}{=} q^{p-s}\cdot {n-s-1\choose p-s}_q.\]
Using the fact that $m_{p-1}(q)-1 = q\cdot {n-p\choose 1}_q$, we obtain
\[\sum_{j\geq 0} [H^j_{\ol{O}_{p-1}}(Q_p)] \cdot q^j = \sum_{s=0}^{p-1} [Q_s] \cdot q^{1+2\cdot(p-s)} \cdot {n-s-1\choose p-s}_{q^2}\]
which agrees with (\ref{eq:square-loccoh-Qp}) in the case when $t=p-1$.

For the induction step, we assume that $p\geq t+2$ and consider the short exact sequence
\begin{equation}\label{eq:exact-Qp}
 0\lra D_p \lra Q_p \lra Q_{p-1}\lra 0
\end{equation}
Combining Theorem~\ref{thm:HOt-Qp-vanishing} with Theorem~\ref{thm:square-loccoh-Dp} we obtain
\[H^{j-1}_{\ol{O}_t}(Q_{p-1})=H^j_{\ol{O}_t}(Q_p)=H^j_{\ol{O}_t}(D_p)=0\mbox{ for }j\not\equiv p-t\ (\opmod\ 2).\]
It follows that the long exact sequence in cohomology associated with (\ref{eq:exact-Qp}) splits into short exact sequences 
\begin{equation}\label{eq:ses-QpDpQp-1}
 0 \lra H^{j-1}_{\ol{O}_t}(Q_{p-1})  \lra H^j_{\ol{O}_t}(D_p) \lra H^j_{\ol{O}_t}(Q_p) \lra 0.
\end{equation}
Since the module $H^j_{\ol{O}_t}(Q_{p})$ is a quotient of $H^j_{\ol{O}_t}(D_{p})$, and the latter belongs to $\add(Q)$ by Proposition~\ref{prop:loccoh-Dpsum}, it follows from Lemma~\ref{lem:quot-Qp} that the former also belongs to $\add(Q)$. It is then sufficient to verify that (\ref{eq:square-loccoh-Qp}) holds in $\Gamma_{\D}[q]$. Using (\ref{eq:ses-QpDpQp-1}), Theorem~\ref{thm:square-loccoh-Dp}, and the induction hypothesis we get
\[
\begin{aligned}
\sum_{j\geq 0} [H^j_{\ol{O}_t}(Q_p)]_{\D} \cdot q^j &= \sum_{j\geq 0} [H^j_{\ol{O}_t}(D_p)]_{\D} \cdot q^j - q\cdot \sum_{j\geq 0} [H^j_{\ol{O}_t}(Q_{p-1})]_{\D} \cdot q^j  \\
&= \sum_{s=0}^t [Q_s]_{\D} \cdot \left[q^{(p-t)^2}\cdot m_s(q^2) - q\cdot q^{(p-1-t)^2+2\cdot(p-1-s)}\cdot {n-s-1\choose p-s-1}_{q^2}\cdot{p-s-2\choose p-t-2}_{q^2} \right]
\end{aligned}
\]
Since $1+(p-1-t)^2+2\cdot(p-1-s)=(p-t)^2+2\cdot(t-s)$, in order to prove (\ref{eq:square-loccoh-Qp}) it suffices to check that
\begin{equation}\label{eq:id-msq-bins}
m_s(q) - q^{t-s}\cdot{n-s-1\choose p-s-1}_{q}\cdot{p-s-2\choose p-t-2}_{q} = q^{p-s} \cdot {n-s-1\choose p-s}_{q} \cdot {p-s-1\choose p-t-1}_{q}.
\end{equation}
When $s=t$, we have ${p-s-2\choose p-t-2}_{q}={p-s-1\choose p-t-1}_{q}=1$ by (\ref{eq:qbin-is-bin}), so (\ref{eq:id-msq-bins}) amounts to the equality
\[{n-t\choose p-t}_q - {n-t-1\choose p-t-1}_{q} = q^{p-t} \cdot {n-t-1\choose p-t}_{q}\]
which follows from (\ref{eq:pascal}). When $s<t$ we get ${p-s-2\choose p-t-2}_{q}={p-s-2\choose t-s}_q$ and ${p-s-1\choose p-t-1}_{q}={p-s-1\choose t-s}_q$ using (\ref{eq:qbin-is-bin}), so we can rewrite (\ref{eq:id-msq-bins}) as
\[{p-s-1\choose t-s}_q\cdot\left[{n-s\choose p-s}_q - q^{p-s}\cdot{n-s-1\choose p-s}_q \right] = \left[q^{t-s}\cdot{p-s-2\choose t-s}_q + {p-s-2\choose t-s-1}_q \right] \cdot {n-s-1\choose p-s-1}_q\]
which follows by applying (\ref{eq:pascal}) to both sides of the equation.
\end{proof}

\subsection{The proof of Theorem~\ref{thm:lyub-square}}\label{subsec:proof-lyub-square}

If $p=n-1$ then $\ol{O}_{n-1}$ is a hypersurface so its only non-zero Lyubeznik number is $\ll_{n^2-1,n^2-1}(R^{(n-1)})=1$ (see \cite[Section~4]{lyub-survey}). We assume that $p\leq n-2$ and get as in Section~\ref{subsec:proof-lyub-non-square} that
\[
\begin{aligned}
L_p(q,w) &= \sum_{i,j\geq 0} \scpr{H^i_{O_0}\left(H^{n^2-j}_{\ol{O}_p}(S)\right)}{D_0}_{\D}\cdot q^i\cdot w^j \\
&= \sum_{i\geq 0}\left[\sum_{s=0}^p \scpr{H^i_{O_0}(Q_s)}{D_0}\cdot q^i \cdot\left(\sum_{j\geq 0}\scpr{H^{n^2-j}_{\ol{O}_p}(S)}{D_s-D_{s+1}}_{\D}\cdot w^j\right)\right]\\
\end{aligned}
\]
where we used the fact that the groups $H^{n^2-j}_{\ol{O}_p}(S)$ belong to $\add(Q)$, and that the multiplicity of $Q_s$ as a summand in $M\in\add(Q)$ can be computed using (\ref{eq:Qp=sum-Ds}) by the formula $\scpr{M}{D_{s}-D_{s+1}}_{\D}$. We obtain that
\[
\begin{aligned}
L_p(q,w) &= \sum_{s=0}^p \scpr{H_0^{\D}(Q_s;q)}{D_0}_{\D} \cdot\scpr{H_p^{\D}(S;w^{-1})}{D_{s}-D_{s+1}}_{\D} \cdot w^{n^2} \\
&\overset{(\ref{eq:square-loccoh-Qp}),(\ref{eq:HDt-S})}{=} \sum_{s=0}^p q^{s^2+2s}\cdot {n-1\choose s}_{q^2}\cdot w^{-(n-p)^2}\cdot\left[{n-1-s\choose p-s}_{w^{-2}}-{n-2-s\choose p-s-1}_{w^{-2}}\right]\cdot w^{n^2}
\end{aligned}
\]
Using (\ref{eq:pascal}) we have 
\[{n-1-s\choose p-s}_{w^{-2}}-{n-2-s\choose p-s-1}_{w^{-2}}=w^{-2\cdot(p-s)}\cdot{n-2-s\choose p-s},\] 
and combining this with (\ref{eq:qbin-inv}) it follows that in order to prove (\ref{eq:lyub-square}) it suffices to verify the identity
\[p^2+2p+s\cdot(2n-2p-2) = -(n-p)^2-2\cdot(p-s)- 2(p-s)\cdot(n-2-p) + n^2\]
which follows again by inspection.

\section*{Acknowledgements}
We are grateful to the organizers of the conference ``Local Cohomology in Commutative Algebra and Algebraic Geometry" at University of Minnesota, where this work was initiated. Raicu acknowledges the support of the Alfred P. Sloan Foundation, and of the National Science Foundation Grant No.~1600765.

\end{document}